\nonstopmode
\documentclass{ps}     

\usepackage{amsmath,amssymb,amsfonts}
\usepackage[numbers]{natbib}               
\usepackage{graphicx}               
\usepackage{color}                  
\usepackage[utf8]{inputenc}
\usepackage[T1]{fontenc} 
\usepackage[absolute]{textpos} 
\usepackage{multicol}
\usepackage{amsthm}
\usepackage{array}

\newcommand{\beq} {\begin{eqnarray*}}
\newcommand{\eeq} {\end{eqnarray*}}

\newcommand{\tbf} {\textbf}

\newcommand{\cvloi}{\overset{\mathcal{L}}{\underset{N\to\infty}{\longrightarrow}}}
\providecommand{\abs}[1]{\left\lvert#1\right\rvert}
\providecommand{\norm}[1]{\left\Vert#1\right\Vert}
\def \R{\mathbb{R}}

\def \N{\mathbb{N}}

\def \E{\mathbb{E}}
\def \P{\mathbb{P}}
\def \Var{\hbox{{\rm Var}}}

\def \Cov{\hbox{{\rm Cov}}}

\newcommand{\ud}{\mathrm{d}}
\newcommand{\Corr}{\mathrm{Corr}}

\newcolumntype{M}[1]{>{\raggedright}m{#1}}\usepackage{array}

\definecolor{orange}{rgb}{1,0.5,0}
\definecolor{vert}{rgb}{0.2,0.8,0.7}
\definecolor{bleu}{rgb}{0.4,0.4,1}

\newtheorem{theorem}{Theorem}[section]

\newtheorem{prop}[theorem]{Proposition}

\newtheorem{lemma}[theorem]{Lemma}
\newtheorem{rem}[theorem]{Remark}

\begin{document}
\title{Asymptotic normality and efficiency of two Sobol index estimators}
\author{Alexandre Janon}
\address{Laboratoire Jean Kuntzmann, Université Joseph Fourier, INRIA/MOISE, 51 rue des Mathématiques, BP 53, 38041 Grenoble cedex 9, France}
\author{Thierry Klein}
\address{Laboratoire de Statistique et Probabilités, Institut de Mathématiques
		Université Paul Sabatier (Toulouse 3) 31062 Toulouse Cedex 9, France}
\author{Agn\`es Lagnoux}
\sameaddress{2}
\author{Ma\"elle Nodet}
\sameaddress{1}
\author{Cl\'ementine Prieur}
\sameaddress{1}
\begin{abstract}
Many mathematical models involve input parameters, which are not precisely known. Global sensitivity analysis aims to identify the parameters whose uncertainty has the largest impact on the variability of a quantity of interest (output of the model). One of the statistical tools used to quantify the influence of each input variable on the output is the Sobol sensitivity index. We consider the statistical estimation of this index from a finite sample of model outputs: we present two estimators and state a central limit theorem for each. We show that one of these estimators has an optimal asymptotic variance. We also generalize our results to the case where the true output is not observable, and is replaced by a noisy version.
\end{abstract}
\begin{resume}
De nombreux modèles mathématiques font intervenir plusieurs paramètres qui ne sont pas tous connus précisément. L'analyse de sensibilité globale se propose de sélectionner les paramètres d'entrée dont l'incertitude a le plus d'impact sur la variabilité d'une quantité d'intérêt (sortie du modèle). Un des outils statistiques pour quantifier l'influence de chacune des entrées sur la sortie est l'indice de sensibilité de Sobol. Nous considérons l'estimation statistique de cet indice à l'aide d'un nombre fini d'échantillons de sorties du modèle: nous présentons deux estimateurs de cet indice et énonçons un théorème central limite pour chacun d'eux. Nous démontrons que l'un de ces deux estimateurs est optimal en terme de variance asymptotique. Nous généralisons également nos résultats au cas o\`u la vraie sortie du modèle n'est pas observée, mais o\`u seule une version dégradée (bruitée) de la sortie est disponible.
\end{resume}
\maketitle

\section*{Introduction}
Many mathematical models encountered in applied sciences involve a large number of poorly-known parameters as inputs. It is important for the practitioner to assess the impact of this uncertainty on the model output. An aspect of this assessment is sensitivity analysis, which aims to identify the most sensitive parameters, that is, parameters having the largest influence of the output. In global stochastic sensitivity analysis (see for example \cite{saltelli-sensitivity} and references therein) the input variables are assumed to be  independent random variables.
Their probability distributions  account for the practitioner's belief about the input uncertainty. This turns the model output into a random variable, whose total variance can be split down into different partial variances (this is the so-called Hoeffding decomposition, see \cite{van2000asymptotic}). Each of these partial variances measures the uncertainty on the output induced by each input variable uncertainty. By considering the ratio of each partial variance to the total variance, we obtain a measure of importance for each input variable that is called the \emph{Sobol index} or \emph{sensitivity index} of the variable \cite{sobol1993}; the most sensitive parameters can then be identified and ranked as the parameters with the largest Sobol indices. 

Once the Sobol indices have been defined, the question of their effective computation or estimation remains open. In practice, one has to estimate (in a statistical sense) those indices using a finite sample (of size typically in the order of hundreds of thousands) of evaluations of model outputs \cite{helton2006survey}. Indeed, many Monte Carlo or quasi Monte Carlo approaches have been developed by the experimental sciences and engineering communities. This includes the FAST methods (see for example \cite{cukier1978nonlinear}, \cite{tissot2010bias} and references therein) and 
the Sobol pick-freeze (SPF) scheme  (see \cite{sobol1993,sobol2001global}). In SPF a Sobol index is viewed as the regression coefficient between the output of the model and its pick-freezed replication. This replication is obtained by holding the value of the variable of interest (frozen variable) and by sampling the other variables (picked variables). The sampled replications are then combined to produce an estimator of the Sobol index. In this paper we study very deeply this Monte Carlo method in the general framework where one or more variables can be frozen. This allows to define sensitivity indices with respect to a general random input living in a probability space (groups of variables, random vectors, random processes...). In this work, we study and compare two Sobol index estimators based on the SPF scheme; the first estimator, denoted by $S_N^X$, is well-known, the second, denoted by $T_N^X$ has been introduced in \cite{Monod2006}. For both estimators, we show convergence and give the rate of convergence; we also show that $T_N^X$ is optimal (in terms of asymptotic variance) amongst regular estimators which are functions of the pick-freezed replications -- this feature is called \emph{asymptotic efficiency} and is a generalization of the notion of minimum variance unbiased estimator (see  \cite{van2000asymptotic} chapters 8 and 25 or \cite{ibragimov1981statistical} for more details).

The SPF method requires many (typically, around one thousand times the number of input variables) evaluations of the model output. In many interesting cases, an evaluation of the model output is made by a complex computer code (for instance, a numerical partial differential equation solving algorithm) whose running time is not negligible (typically in the order of a second or a minute) for one single evaluation. When thousands of such evaluations have to be made, one generally replaces the original {\it exact} model by a faster-to-run \emph{metamodel} (also known in the literature as \emph{surrogate model} or \emph{response surface} \cite{box1987empirical}) which is an approximation of the true model. Well-known metamodels include Kriging \cite{sant:will:notz:2003}, polynomial chaos expansion \cite{sudret2008global} and reduced bases \cite{nguyen2005certified,janon2011certified}, to name a few. When a metamodel is used, the estimated Sobol indices are tainted by a twofold error: \emph{sampling error}, due to the replacement of the original, infinite population of all the possible inputs by a finite sample, and \emph{metamodel error}, due to the replacement of the original model by an approximative metamodel.  

The goal of this paper is to study the asymptotic behavior of these two errors on Sobol index estimation in the double limit where the sample size goes to infinity and the metamodel converges to the true model. Some work has been done on the non-asymptotic error quantification in Sobol index estimation in earlier papers \cite{storlie2009implementation,marrel2009calculations,janon2011confidence} by means of confidence intervals which account for both sampling and metamodel errors. In this paper, we give necessary and sufficient conditions on the rate of convergence of the metamodel to the exact model for asymptotic normality of a natural Sobol index estimator to hold. The asymptotic normality allows us to produce asymptotic confidence intervals in order to assess the quality of our estimation. We also give sufficient conditions for a metamodel-based estimator to be asymptotically efficient. Asymptotic efficiency of an other Sobol index estimator has already been considered in \cite{da2008efficient}. In this work, the authors were interested in the asymptotic efficiency for local polynomial estimates of Sobol indices. Our approach proposes an estimator 
which has a simpler form, is less computationally intensive and is more precise in practice. Moreover, we derive results also in the case where the full model is replaced by a metamodel. 

This paper is organized as follows: in the first section, we set up the notation, review the definition of Sobol indices and give two estimators 
 of interest. In the second section, we prove asymptotic normality and asymptotic efficiency when the sample of outputs comes from the true model. These two properties are generalized in the third section where metamodel error is taken into account. The fourth section gives numerical illustrations on benchmark models and metamodels.

\section{Definition and estimation of Sobol indices}

\subsection{Exact model}

The output $Y \in \R$ is a function of independent random input variables $X \in \R^{p_1}$ and $Z \in \R^{p_2}$. In other words, $Y$ and $(X,Z)$ are linked by the relation 
\begin{equation}\label{def:lien_boite_noire}
Y=f(X,Z)
\end{equation}
where $f$ is a deterministic function defined on $\mathcal P \subset \R^{p_1+p_2}$. We denote by $p=p_1+p_2$ the total number of inputs of $f$.

In the paper $X'$ will denote an independent copy of $X$. We also write $Y^X=f(X,Z')$.

We assume that $Y$ is square integrable and non deterministic ($\Var(Y) \neq 0$). We are interested in the following Sobol index: \begin{equation}\label{def:sobol}
	S^X=\frac{\Var\left(\mathbb{E}(Y|X)\right)}{\Var(Y)} \in [0,1].
\end{equation}

This index quantifies the influence of the $X$ input on the output $Y$: a value of $S^X$ that is close to $1$ indicates that $X$ is highly influential on $Y$. 

\begin{rem} All the results in this paper readily apply when $X$ is multidimensional. In this case, $S^X$ is usually called the \emph{closed sensitivity index} of $X$ (see \cite{saltelli2004practice}). 

Note that this separation between the input variables can be made without loss of generality, when one estimates Sobol indices independently. An ongoing work treats the case of joint Sobol index estimation.

\end{rem}

\subsection{Estimation of $S^X$}
The next lemma shows how to express $S^X$ using covariances. This will lead to a natural estimator which has already been considered in \cite{homma1996importance}. 

\begin{lemma}\label{lemma:cov}
Assume that the random variables $Y$ and $Y^X$ are square integrable.  Then
\[ \Var(\mathbb{E}(Y|X))=\Cov(Y,Y^X). \]
In particular
\begin{equation}\label{sobol_cov}
S^X=\frac{\Cov\left(Y,Y^X\right)}{\Var(Y)}.
\end{equation}
\end{lemma}

\begin{rem}
Using a classical regression result, we see that 
\begin{align}
S^X&=\underset{a \in \R}{\mathrm{argmin}} \left\{\mathbb{E}\left((Y^X-\mathbb{E}(Y^X))-a(Y-\mathbb{E}(Y))\right)^2\right\}.
\end{align}
\end{rem}

\tbf{A first estimator. } In view of Lemma \ref{lemma:cov}, we are now able to define a first natural estimator of $S^{X}$ (all sums are taken for $i$ from $1$ to $N$): 
\begin{equation}
\label{e:sxn}
S^X_N = \frac{ \frac{1}{N} \sum  Y_i   Y_i^X - \left(\frac{1}{N} \sum Y_i\right) \left(\frac{1}{N}\sum  Y_i^X\right) }{ \frac{1}{N}\sum  Y_i^2 - \left( \frac{1}{N} \sum  Y_i \right)^2 }, 
\end{equation}
where, for $i=1,\ldots,N$:
\[ Y_i=f(X_i,Z_i), \;\;\; Y_i^X=f(X_i,Z_i'), \]
and $\{ (X_i,Z_i) \}_{i=1,\ldots,N}$ and $\{ (X_i,Z_i') \}_{i=1,\ldots,N}$ are two independent and identically distributed (i.i.d.) samples of the distribution of $(X,Z)$, with $\{Z_i\}_i$ independent of $\{Z_i'\}_i$.

This estimator has been considered in \cite{homma1996importance}, where it has been shown to be a practically efficient estimator.\\

\tbf{A second estimator. } We can take into account the observation of $\{ Y_i^X \}_{1 \leq i \leq N}$ to make an estimation of $\E(Y)$ and $\Var(Y)$ which is expected to perform better than any other based on $\{ Y_i \}_{1 \leq i \leq N}$ only. We propose the following estimator:
\begin{equation}
\label{e:txn}
T^X_N = \frac{ \frac{1}{N} \sum  Y_i   Y_i^X - \left(\frac{1}{N} \sum \Big[\frac{Y_i+Y_i^X}{2}\Big]\right)^2 }{ \frac{1}{N}\sum  \Big[\frac{Y_i^2+(Y_i^X)^2}{2}\Big] - \left( \frac{1}{N} \sum  \Big[\frac{Y_i+Y_i^X}{2}\Big] \right)^2 }.
\end{equation}

This estimator has been introduced in \cite{Monod2006}. We will clarify what we mean when saying that $T_N^X$ performs better than $S_N^X$ in Proposition \ref{prop:varless}, Section \ref{ss:effic} and Subsection \ref{ss:res1}.

\begin{rem}
Note that the empirical variances in $S^X_N$ and $T^X_N$ can be rewritten as:
\begin{align}
S^X_N &= \frac{ \sum (Y_i - \overline Y)(Y_i^X - \overline {Y^X}) }{ \sum (Y_i - \overline Y)^2 } \\
T^X_N &= \frac{ \sum ( Y_i - \overline {Y_2} ) ( Y_i^X - \overline{Y_2} ) }{ \sum \left( \frac{Y_i+Y_i^X}{2} - \overline{Y_2} \right)^2 }
\end{align}
where:
\[
\overline Y = \frac1N \sum Y_i, \; \overline{Y^X} = \frac1N \sum Y_i^X, \; \overline{Y_2} = \frac{\overline Y+\overline{Y^X}}{2}.
\]
The use of these formulae enables greater numerical stability (ie., less error due to round-offs). The Kahan compensated summation algorithm \cite{kahan1965pracniques} may also be used on these sums. However, we will use definitions \eqref{e:sxn} and \eqref{e:txn} for the mathematical analysis of $S^X_N$ and $T^X_N$. This analysis is of course independent of the way the estimators are numerically computed in practice.

\end{rem}

\section{Asymptotic properties: exact model}

\subsection{Consistency and asymptotic normality}\label{ssec:cons}
Throughout all the paper, we denote by $\mathcal N_k(\mu,\Sigma)$ the $k$-dimensional Gaussian distribution with mean $\mu$ and covariance matrix $\Sigma$, and, given any sequence of random variables $\{ R_n \}_{n\in\N}$, we note
\[ \overline R_N = \frac{1}{N} \sum_{n=1}^N R_n. \]
\begin{prop}[Consistency]\label{prop:cons}
We have:
\begin{equation}\label{cv_a}
S^X_N\overset{a.s. }{\underset{N\to\infty}{\longrightarrow}} S^X
\end{equation}
\begin{equation}\label{cv_a_2}
T^X_N\overset{a.s. }{\underset{N\to\infty}{\longrightarrow}} S^X.
\end{equation}
\end{prop}

\begin{proof}
	The result is a straightforward application of the strong law of large numbers and that $\E(Y)=\E(Y^X)$ and $\Var(Y)=\Var(Y^X)$.
\end{proof}

\begin{prop}[Asymptotic normality] \label{prop:an1} Assume that $\E(Y^4)<\infty$. Then\\
		\begin{equation}\label{cv_n_a}
\sqrt{N} \left( S^X_N - S^X \right)
\overset{\mathcal{L}}{\underset{N\to\infty}{\rightarrow}}
\mathcal{N}_1\left(0, \sigma_S^2 \right)  
\end{equation}
and
 \begin{equation}\label{cv_n_a_2}
\sqrt{N} \left( T^X_N - S^X \right)
\overset{\mathcal{L}}{\underset{N\to\infty}{\rightarrow}}
\mathcal{N}_1\left(0, \sigma_T^2\right)
\end{equation}
where
\[
\sigma_S^2=\frac{ \Var\left( ( Y-\E(Y) ) \left[ ( Y^X - \E(Y) ) - S^X ( Y - \E(Y) ) \right] \right)}{\left( \Var(Y) \right)^2},
\]
\[
\sigma_T^2=	\frac{\Var\left( (Y-\E(Y))(Y^X-\E(Y)) - S^X/2 \left( (Y-\E(Y))^2+(Y^X-\E(Y))^2 \right) \right)}{(\Var(Y))^2}. 
\]
\end{prop}

\begin{prop} \label{prop:varless}
The asymptotic variance of $T_N^X$ is always less than or equal to the asymptotic variance of $S_N^X$, with equality if and only if $S^X=0$ or $S^X=1$.
\end{prop}
To prove this Proposition, we need the following immediate Lemma:
\begin{lemma} \label{lemm:exch}
	$Y$ and $Y^X$ are exchangeable random variables, ie. $(Y, Y^X) \buildrel \mathcal{L} \over = (Y^X, Y)$.
\end{lemma}

\subsection{Asymptotic efficiency} \label{ss:effic}

In this section we study the asymptotic efficiency of $S_N^X$ and $T_N^X$. This notion (see \cite{van2000asymptotic}, Section 25 for its definition) extends the notion of Cramér-Rao bound to the semiparametric setting and enables to define a criteria of optimality for estimators, called asymptotic efficiency.

Let $\mathcal{P}$ be the set of all cumulative distribution functions (cdf) of exchangeable random vectors in $L^2(\R^2)$. It is clear that the cdf $Q$ of a random vector of $L^2(\R^2)$ is in $\mathcal{P}$ if and only if $Q$ is symmetric: 
\[ Q(a,b)=Q(b,a) \;\;\; \forall (a,b)\in\R^2. \]

Let $P$ be the cdf of $(Y,Y^X)$. We have $P \in \mathcal P$ thanks to Lemma \ref{lemm:exch}.

\begin{prop}[Asymptotic efficiency]\label{prop:ae}
$\{ T^X_N \}_N$ is asymptotically efficient for estimating $S^X$ for $P\in\mathcal{P}$.
\end{prop}

We will use the following Lemma, which is also of interest in its own right:
\begin{lemma}[Asymptotic efficiency in $\mathcal{P}$]\label{lemm:eff}
\begin{enumerate} 
 \item Let $\Phi_1: \R \rightarrow \R$ be a function in $L^2(P)$. The sequence of estimators $\left\{\Phi_N^1\right\}_N$ given by:
 \[ \Phi_N^1 = \frac{1}{N} \sum \frac{ \Phi_1(Y_i) + \Phi_1(Y_i^X) }{2} \]
 is asymptotically efficient for estimating $\E (\Phi_1(Y)) $ for $P\in\mathcal P$.

 \item Let $\Phi_2: \R^2 \rightarrow \R$ be a symmetric function in $L^2(P)$. The sequence $\left\{\Phi_N^2\right\}_N$ given by:
 \[ \Phi_N^2 = \frac{1}{N} \sum \Phi_2 \left( Y_i, Y_i^X \right) \]
 is asymptotically efficient for estimating $\E(\Phi_2(Y,Y^X)) $ for $P\in\mathcal P$.

\end{enumerate}
\end{lemma}

\section{Asymptotic properties: metamodel}
\subsection{Metamodel-based estimation}
As said in the introduction, we often are in a situation where the exact output $f$ is too costly to be evaluated numerically (thus, $Y$ and $Y^X$ are not observable variables in our estimation problem) and has to be replaced by a metamodel $\widetilde{f}$, which is a faster to evaluate approximation of $f$. We view this approximation as a perturbation of the exact model by some function $\delta$: 
\[ \widetilde Y=\widetilde f(X,Z)=f(X,Z)+\delta, \]
where the perturbation $\delta=\delta(X,Z,\xi)$ is also a function of a random variable $\xi$ independent from $X$ and $Z$.

We also define, as before 
\[\widetilde Y^X=\widetilde f(X,Z'). \] 
Assuming again that $\widetilde Y$ is non deterministic and in $L^2$, we can consider the following Sobol index, with respect to the metamodel:
\begin{equation}\label{def:sobol_pert}
	\widetilde S^X = \frac{\Var (\E(\widetilde Y|X))}{\Var (\widetilde Y)}
\end{equation}
and its estimators:
\begin{align} \label{e:estimmeta} \widetilde S^X_N &= \frac{ \frac{1}{N} \sum \widetilde Y_i  \widetilde Y_i^X - \left(\frac{1}{N} \sum\widetilde Y_i\right) \left(\frac{1}{N}\sum \widetilde Y_i^X\right) }{ \frac{1}{N}\sum \widetilde Y_i^2 - \left( \frac{1}{N} \sum \widetilde Y_i \right)^2 }\end{align}
\begin{equation} \label{e:estimmeta2} 
\widetilde T^X_N = \frac{ \frac{1}{N} \sum \widetilde Y_i  \widetilde Y_i^X - \left(\frac{1}{N} \sum  \Big[\frac{\widetilde Y_i+\widetilde Y_i^X}{2}\Big]\right)^2 }{\frac{1}{N}\sum \Big[\frac{\widetilde Y_i^2+ (\widetilde Y_i^X)^2}{2}\Big] - \left( \frac{1}{N} \sum  \Big[\frac{\widetilde Y_i+\widetilde Y_i^X}{2}\Big] \right)^2 }.
\end{equation}
The goal of this section is to give sufficient conditions on the perturbation $\delta$ for $\widetilde S^X_N$ and $ \widetilde T^X_N$ to satisfy asymptotic normality (Subsection \ref{ssec:consmeta}), and $\widetilde T^X_N$ to be asymptotically efficient (Subsection \ref{ssec:effmeta}), with respect to the Sobol index of the \emph{true} model $S^X$.

\subsection{Consistency and asymptotic normality} \label{ssec:consmeta}

In the first Subsection (\ref{s:stationnaire}) we suppose that the error term $\delta$ does not depend on $N$. In this case, if the Sobol index of the exact model is different from the Sobol index of the metamodel, then neither consistency nor asymptotic normality are possible. 
In the second subsection (\ref{s:nonstationnaire}), we let $\delta$ depend on $N$ and we give conditions for consistency and asymptotic normality to hold.  

\subsubsection{First case : $\delta$ does not depend on $N$}\label{s:stationnaire}
\begin{rem}\label{prop:pasconsistant}
If $\widetilde {S}^X - S^X\neq0$ then neither $\widetilde S^X_N$ nor $\widetilde T^X_N$ are consistent for estimating $S^X$.

Indeed, we have
\[ \widetilde {S}^X_N - S^X = (\widetilde S^X_N - \widetilde S^X ) + (\widetilde {S}^X - S^X). \]
The first term converges to 0 almost surely by Proposition \ref{prop:cons} applied to $\widetilde{S}^X_N$. However, the second is nonzero by assumption.  The same holds for $\widetilde T^X_N$.

\end{rem}

This remark shows that a naive consideration of the metamodel error (ie., with fixed metamodel) is not satisfactory for an asymptotic justification of the use of a metamodel. More specifically, it is impossible to have asymptotic normality for $\widetilde S^X_N$ and $\widetilde T^X_N$ in any nontrivial case if $\delta$ does not vanish (in some sense) asymptotically. This justifies the consideration of cases where $\delta$ depends on $N$, and this is the object of the next subsection.
 
\subsubsection{Second case : $\Var\;\delta_N$ converges to 0 as $N\rightarrow \infty$}\label{s:nonstationnaire}
We now assume that the perturbation $\delta$ is a function of the sample size $N$. This entails that $\widetilde f$, as well as $\widetilde Y$, $\widetilde Y^X$ and $\widetilde S^X$ depend on $N$. We emphasize this dependence by using the notations $\delta_N$, $\widetilde f_N$, $\widetilde Y_N$, $\widetilde Y^X_N$. We keep, however, using the notations $\widetilde S^X_N$ and $\widetilde T^X_N$ for the estimators of $\widetilde S^X$ defined at \eqref{e:estimmeta} and \eqref{e:estimmeta2}. 

\textbf{Assumption.} We suppose that $\widetilde f_N - f = \delta_N \overset{L^2}{\underset{N\rightarrow+\infty}{\longrightarrow}} c$ for some constant $c$.

\begin{prop}
\label{p:propAA}
	We have $\widetilde S^X \underset{N\rightarrow+\infty}{\longrightarrow} S^X. $
\end{prop}

\begin{prop}\label{prop:as_norm} Assume there exist $s>0$ and $C>0$ such that
	\begin{equation}\label{e:ass4ps} \forall N,\;\; \E\left(\abs{\widetilde Y_N}^{4+s} \right)<C. \end{equation}

	Then
\begin{equation}\label{cv_n_aa}
\sqrt{N}\left(\widetilde S_N^X - \widetilde{S}^X\right)
\cvloi
\mathcal{N}_1\left(0, \sigma_S^2\right) 
\end{equation}
\begin{equation}\label{cv_n_aa2}
\sqrt{N}\left(\widetilde T_N^X - \widetilde{S}^X\right)
\cvloi
\mathcal{N}_1\left(0, \sigma_T^2\right) 
\end{equation}
where $\sigma_S^2$ and $\sigma_T^2$ are the asymptotic variances of $S_N^X$ and $T_N^X$ given, respectively, in \eqref{cv_n_a} and \eqref{cv_n_a_2}.
\end{prop}
We are actually interested in the asymptotic distribution of $\sqrt{N}\left(\widetilde{S}_N^X-S^X\right)$. In the remaining of the subsection, we will show that this convergence depends on the rate of convergence to 0 of $\Var(\delta_N)$.

\begin{theorem}\label{cor:as_norm}
Let:
\[ C_{\delta,N}=2\Var(Y)^{1/2}\left[\Corr(Y,\delta^X_N)-\Corr(Y,Y^X)\Corr(Y,\delta_N)\right]+\Var(\delta_N)^{1/2}\left[\Corr(\delta_N,\delta^X_N)-\Corr(Y,Y^X)\right],  \]
for $\delta^X_N=\delta_N(X,Z')$, and, given any $L^2$ random variables $A$ and $B$ of nonzero variance:
        \[ \Corr(A,B)=\frac{\Cov(A,B)}{\sqrt{\Var A\, \Var B}}. \]

Assume that $C_{\delta,N}$ does not converge to 0.
\begin{enumerate}
	\item If $\Var(\delta_N)=o\left(\frac{1}{N}\right)$, then asymptotic normalities of $\widetilde S_N^X$ and $\widetilde T_N^X$ for $S^X$ hold, i.e. 
        	\begin{equation}
		\label{e:asnorm}
		\sqrt N (\widetilde S_N^X - S^X)  \underset{N\rightarrow+\infty}{\longrightarrow} \mathcal N (0, \sigma_S^2) 
	\end{equation}
         and:       	\begin{equation}
		\label{e:asnorm2}
		\sqrt N (\widetilde T_N^X - S^X)  \underset{N\rightarrow+\infty}{\longrightarrow} \mathcal N (0, \sigma_T^2).
	\end{equation}

\item If $N\Var(\delta_N)\to \infty$, then \eqref{e:asnorm} and \eqref{e:asnorm2}.
\item If $C_{\delta,N}$ converges to a nonzero constant $C$ and  $\gamma\in\R$ so that $\Var(\delta_N)=\frac{\gamma}{CN}+o\left(\frac{1}{N}\right)$, then:
\[ \sqrt{N}(\widetilde S_N^X- S^X) \underset{N\rightarrow+\infty}{\longrightarrow} \mathcal N (\gamma, \sigma_S^2),\]
and:
\[ \sqrt{N}(\widetilde T_N^X- S^X) \underset{N\rightarrow+\infty}{\longrightarrow} \mathcal N (\gamma, \sigma_T^2).\]
\end{enumerate}
\end{theorem}

\begin{rem}
Obviously, if $C_{\delta,N}$ converges to 0, then asymptotic normalities of $\widetilde S_N^X$ and $\widetilde T_N^X$ hold under weaker assumptions on $\Var(\delta_N)$.
\end{rem}

\subsection{Asymptotic efficiency}\label{ssec:effmeta}
\begin{prop}[Asymptotic efficiency for the metamodel]\label{prop:aem}
Assume
	\begin{enumerate}
		\item $\exists s>0, C>0 \text{ s.t. } \forall N, \; \E\left(\big\vert Y\big\vert^{4+s}\right)<C $ and $\E\left(\big\vert\widetilde{Y}\big\vert^{4+s}\right)<C$ ;
		\item $N \Var(\delta_N) \rightarrow 0$ ;
		\item $\sqrt N \E(\delta_N) \rightarrow 0$.
	\end{enumerate}
Then $\left\{ \widetilde T_N^X \right\}$ is asymptotically efficient for estimating $S^X$.
\end{prop}

\begin{rem}
	By Minkowski inequality, the first hypothesis implies $\E(\delta_N^{4+s})<2C^{\frac{1}{4+s}}$ and the asymptotic normality by Lemma \ref{prop:as_norm} and Theorem \ref{cor:as_norm}.
\end{rem}

\section{Numerical illustrations}
In this section, we illustrate the asymptotic results of Sections \ref{ssec:cons} and \ref{ssec:consmeta} when the exact model is the Ishigami function \cite{ishigami1990importance}:
\begin{equation}
	\label{e:ishigam}
	f(X_1,X_2,X_3)=\sin X_1+7 \sin^2 X_2+0.1 X_3^4 \sin X_1 
\end{equation}
for $(X_j)_{j=1,2,3}$ are i.i.d. uniform random variables in $[-\pi;\pi]$. In this case, all the integrability conditions are satisfied (we even have $Y \in L^\infty$).

The Sobol index of $f$ with respect to input variable $X_1$ is $S^X$ defined in \eqref{def:sobol} for $X=X_1$ and $Z=(X_2,X_3)$; we denote it by $S^1$. Similarly, $S^2$ (resp. $S^3$) is $S^X$ obtained taking $X=X_2$ and $Z=(X_1,X_3)$ (resp. $X=X_3$ and $Z=(X_1,X_2)$).

Exact values of these indices are analytically known: 
\[ S^1=0.3139, \;\; S^2=0.4424, \;\; S^3=0. \]
For a sample size $N$, a risk level $\alpha \in ]0;1[$  and for each input variable, a confidence interval for $S^X$ ($S^X$ being one of $S^1$, $S^2$ or $S^3$) of confidence level $1-\alpha$ can be estimated -- using evaluations of the true model $f$ -- by approximating the distribution of $S^X_N$ (or $T^X_N$) by its Gaussian distribution given in Proposition \ref{cv_n_a}, using empirical estimators of the asymptotic variances stated in this Proposition.
 
In the case where only a perturbated model (metamodel) $\widetilde f_N=f+\delta_N$ is available, a confidence interval can still be estimated by using the $\widetilde S^X_N$ (or $\widetilde T^X_N$) estimator. 

Thanks to Proposition \ref{prop:as_norm}, the level of the resulting confidence interval should be close to $1-\alpha$ for sufficiently large values of $N$ if (and only if) $\Var \delta_N$ decreases sufficiently quickly with $N$.

The levels of the obtained confidence interval can be estimated by computing a large number $R$ of confidence interval replicates, and by considering the empirical coverage, that is, the proportion of intervals containing the true index value; it is well known that this empirical coverage strongly converges to the level of the interval as $R$ goes to infinity.

In the next subsections, we present the estimations of the levels of the confidence interval for the Ishigami model \eqref{e:ishigam} using the true model (Subsection \ref{ss:res1}), and,  with various synthetic model perturbations (Subsections \ref{ss:res2} and \ref{ss:res3}), as well as RKHS (Kriging) metamodels (Subsection \ref{ss:rkhsmeta}) and nonparametric regression metamodels (Subsection \ref{ss:regrmeta}). We begin by comparing $S_N$ and $T_N$ on the exact model (Subsection \ref{ss:res1}), then we illustrate the generalization to the metamodel case on the widespread estimator $S_N$; the condition to ensure asymptotic normality in the metamodel is the same for $S_N$ and $T_N$. All simulations have been made with $R=1000$ and $\alpha=0.05$.

\subsection{Exact model} \label{ss:res1}
Figure \ref{f:1} shows the empirical coverage of the asymptotic confidence interval built using the $S_N^X$ estimator, plotted as a function of the sample size $N$. The theoretical level $0.95$ is represented with a dotted line. Figure \ref{f:1bis} does the same using the $T_N^X$ estimator.
\begin{figure}
	\begin{center}
		\includegraphics[scale=.4]{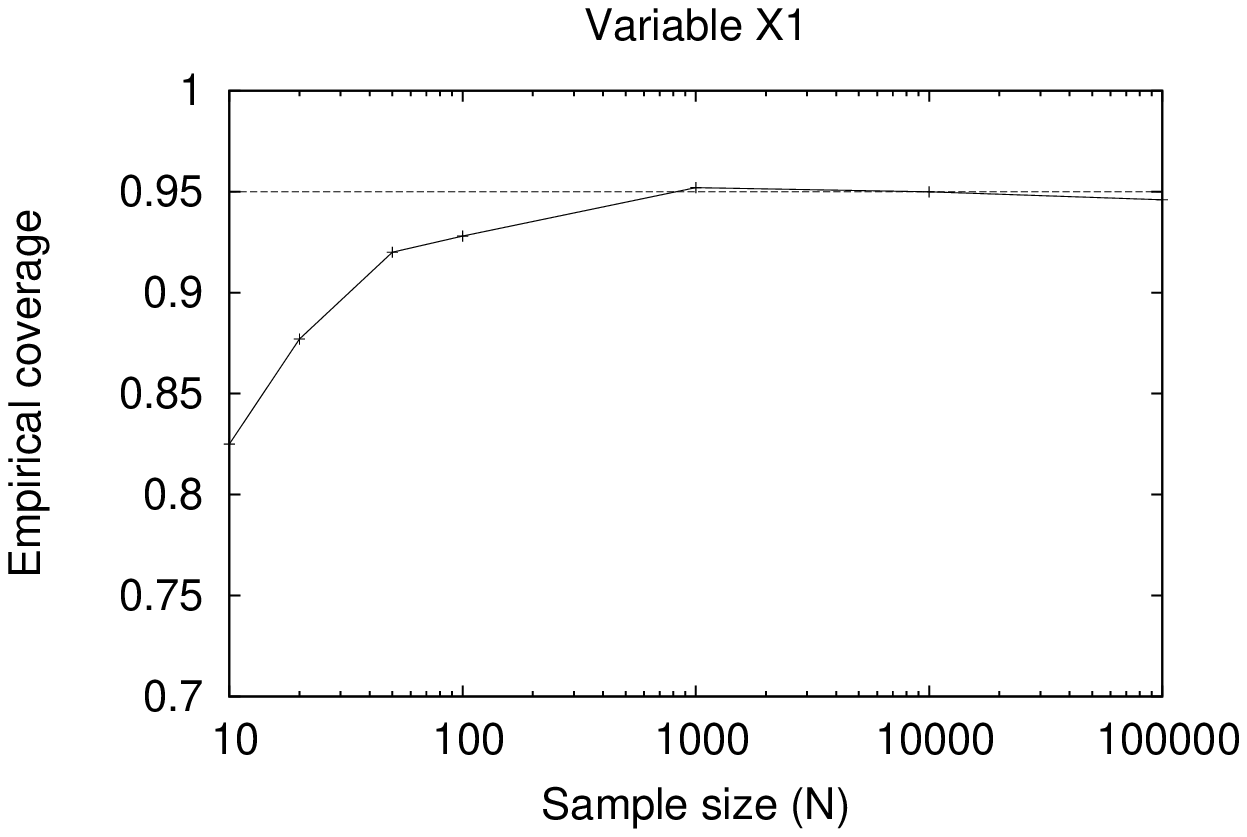}
		\includegraphics[scale=.4]{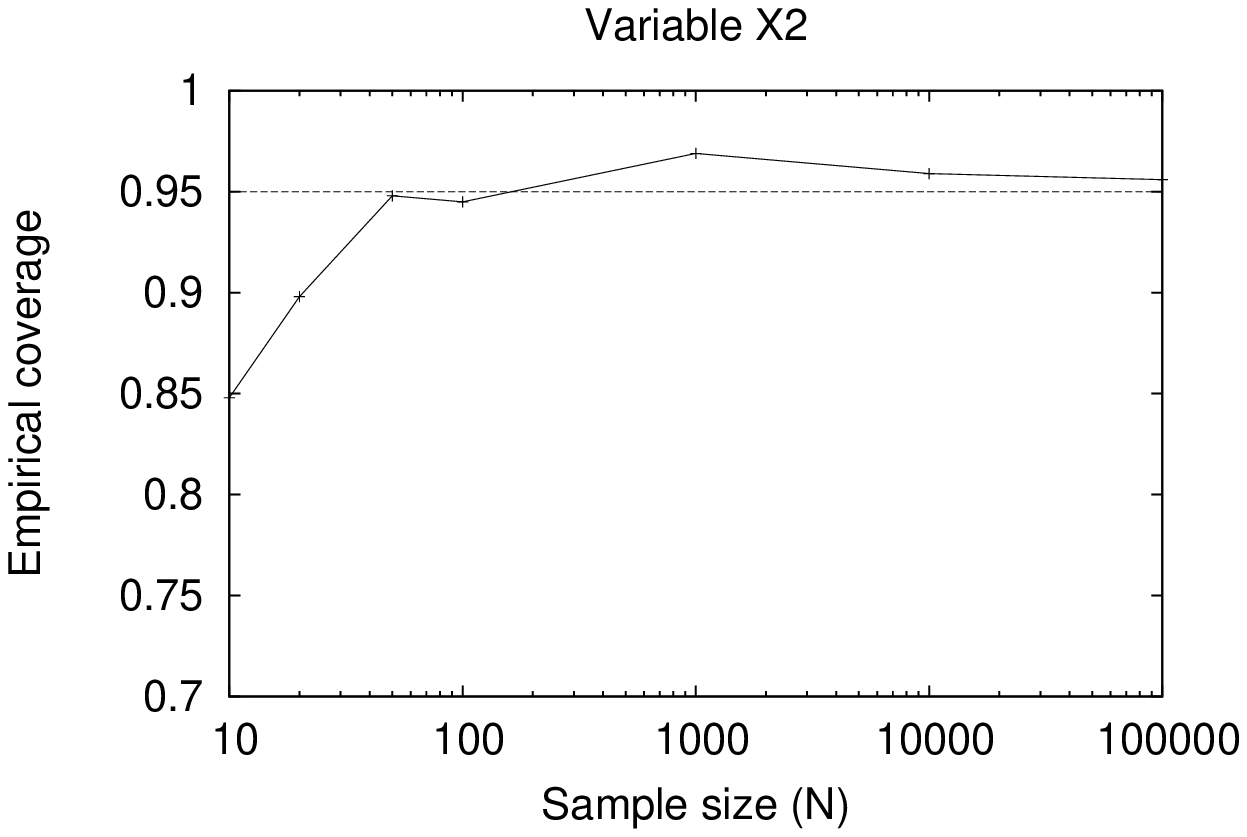}
		\includegraphics[scale=.4]{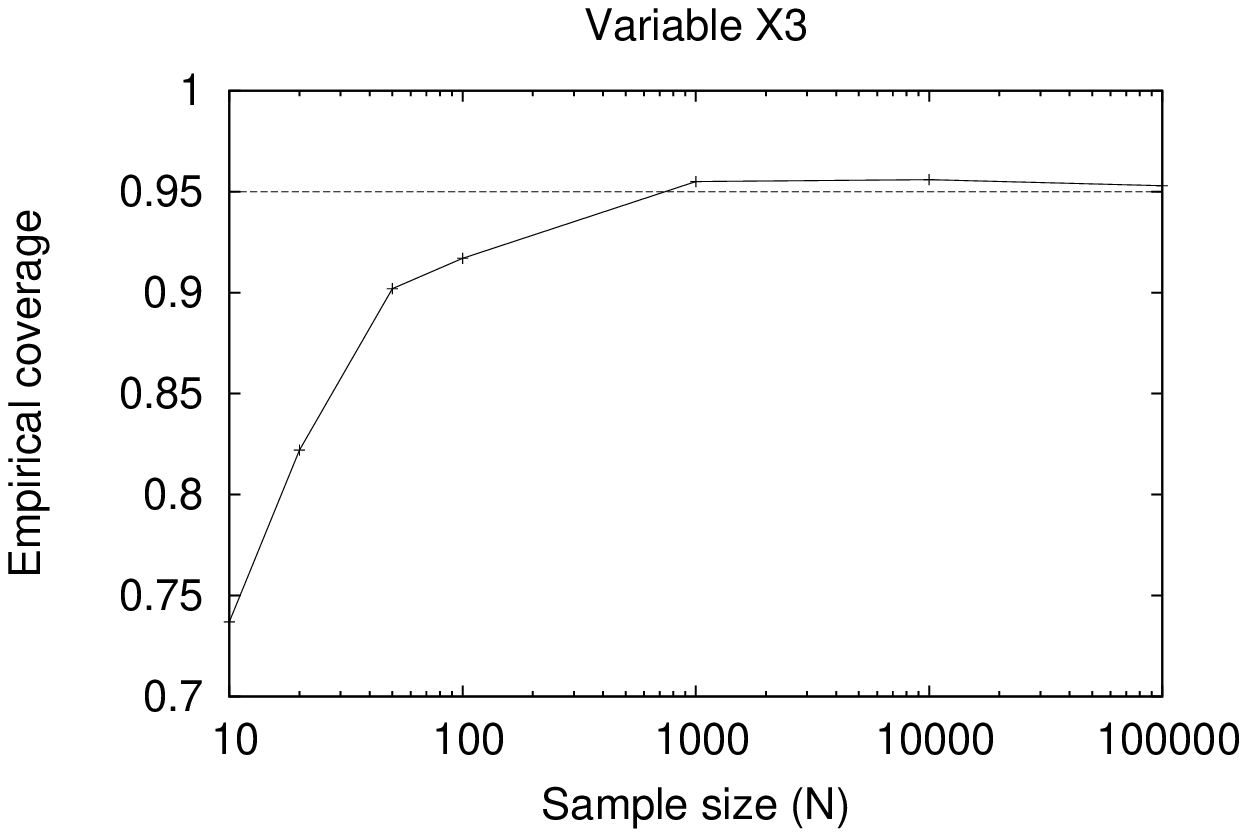}
		\caption{Empirical coverages of asymptotic confidence intervals for $S^1$ (left), $S^2$ (center) and $S^3$ (right), as a function of the sample size. The $S_N$ estimator is used. }
		\label{f:1}
	\end{center}
\end{figure}

\begin{figure}
	\begin{center}
		\includegraphics[scale=.4]{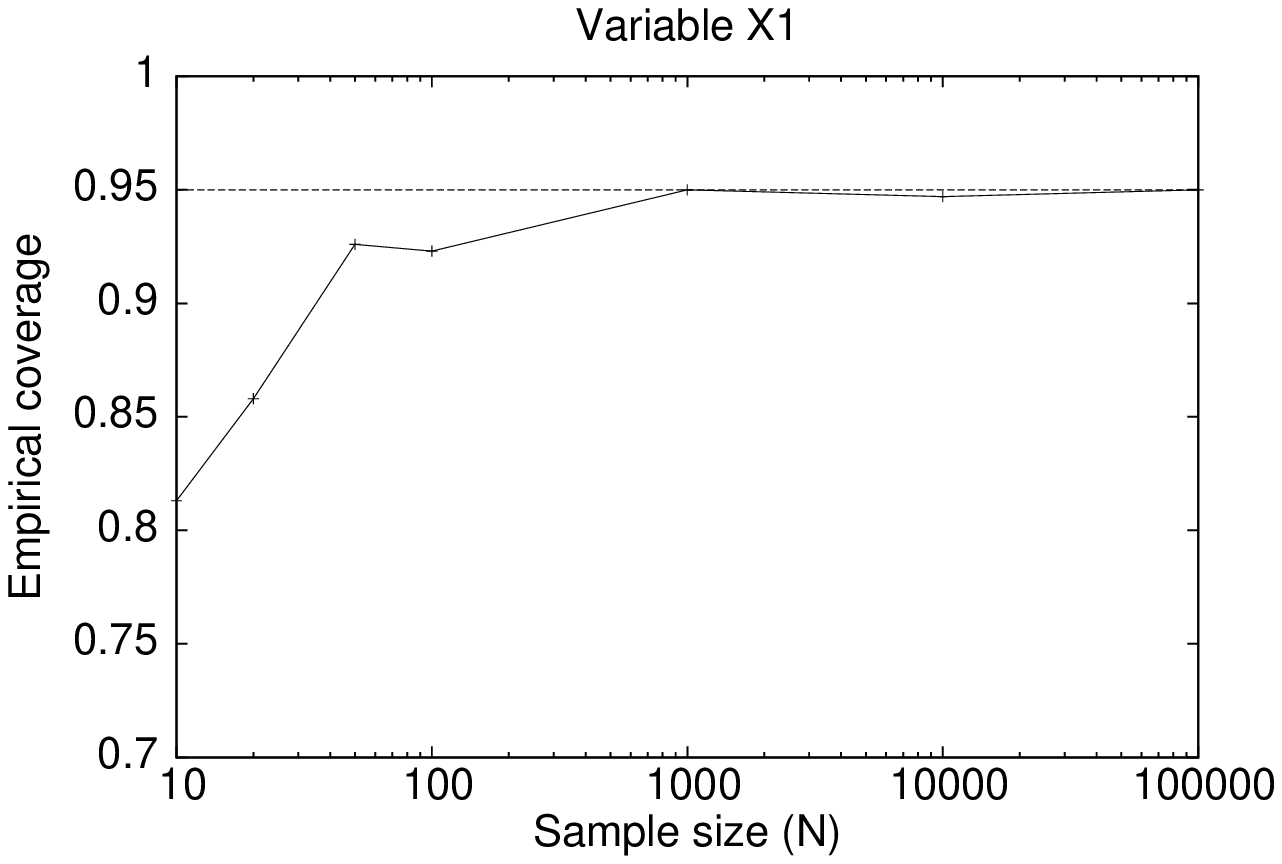}
		\includegraphics[scale=.4]{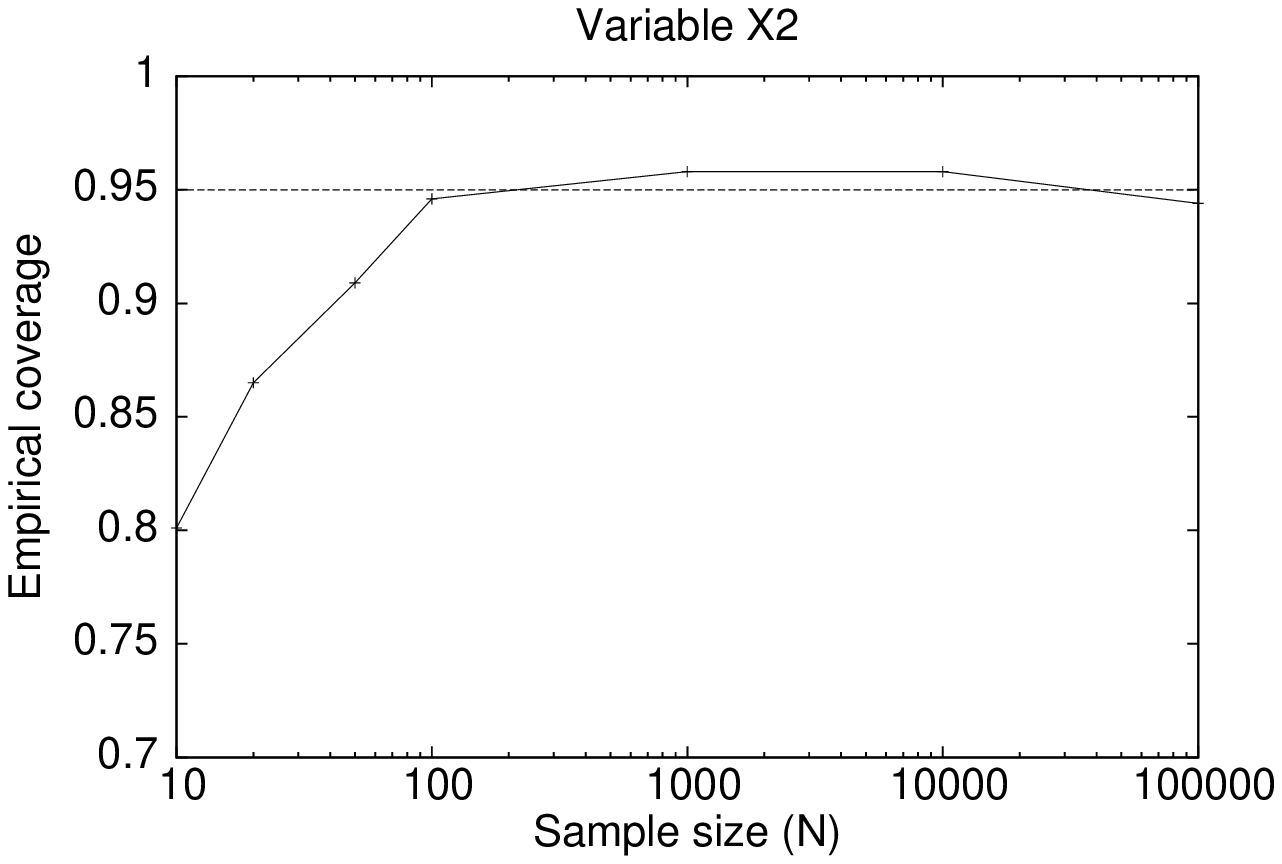}
		\includegraphics[scale=.4]{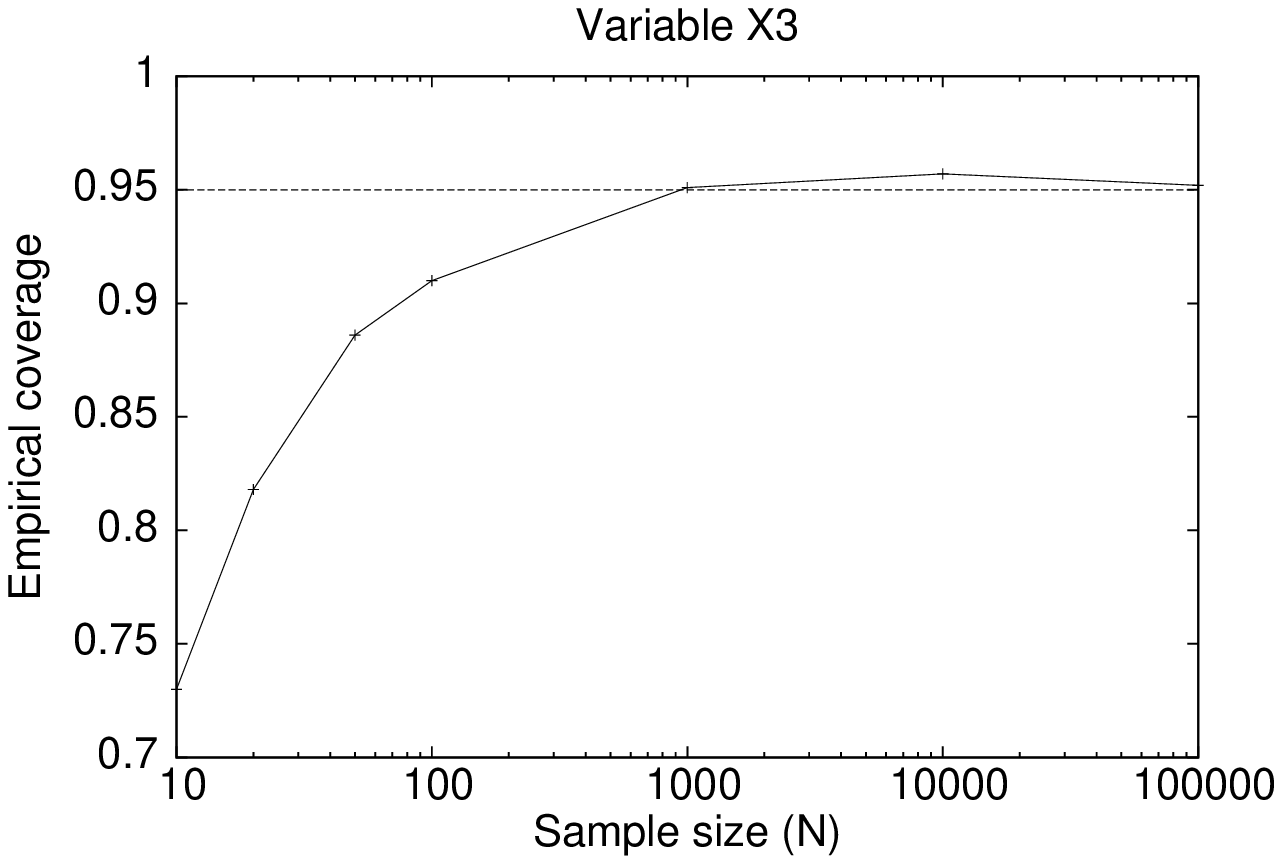}
		\caption{Empirical coverages of asymptotic confidence intervals for $S^1$ (left), $S^2$ (center) and $S^3$ (right), as a function of the sample size (for the exact model). The $T_N$ estimator is used. }
		\label{f:1bis}
	\end{center}
\end{figure}

\begin{figure}
	\begin{center}
		\includegraphics[scale=.4]{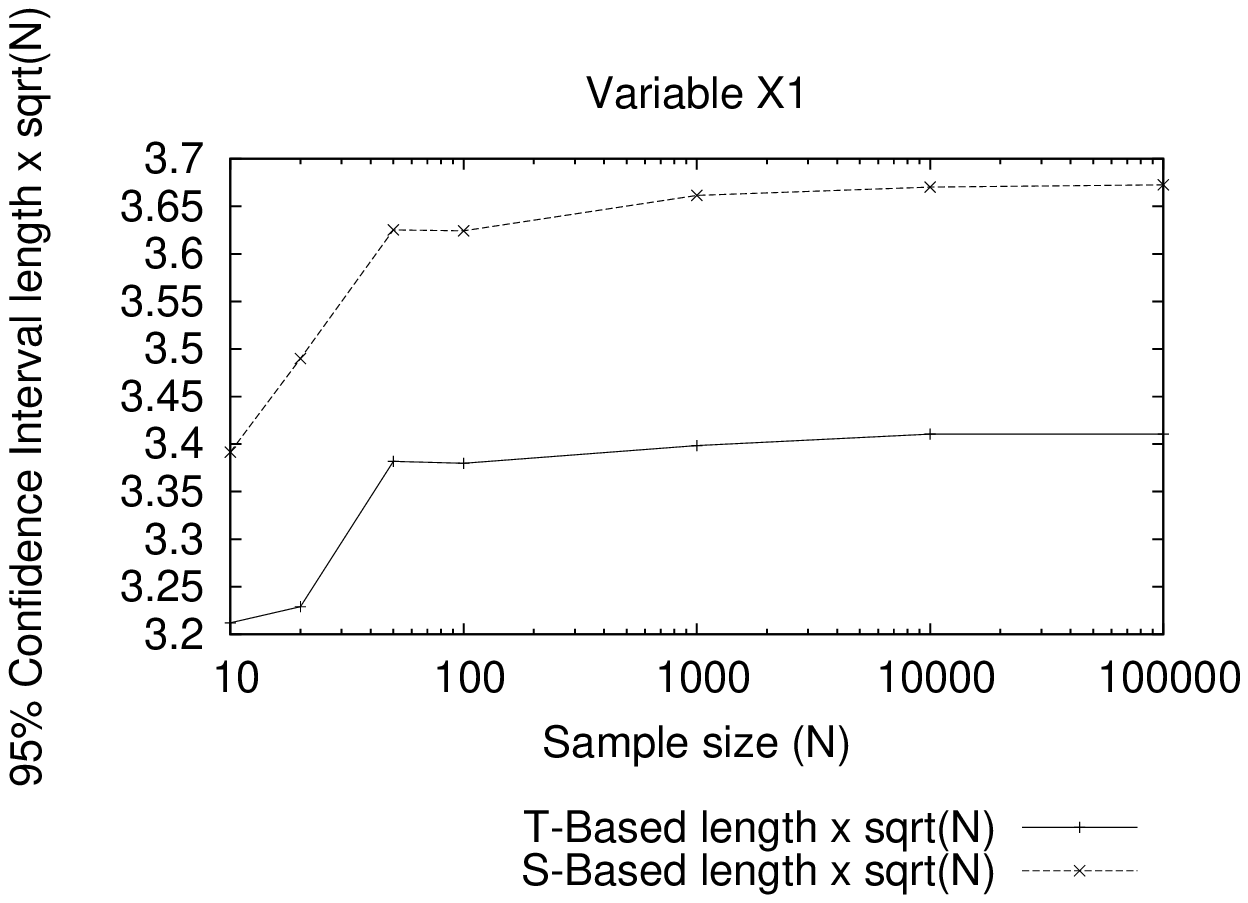}
		\includegraphics[scale=.4]{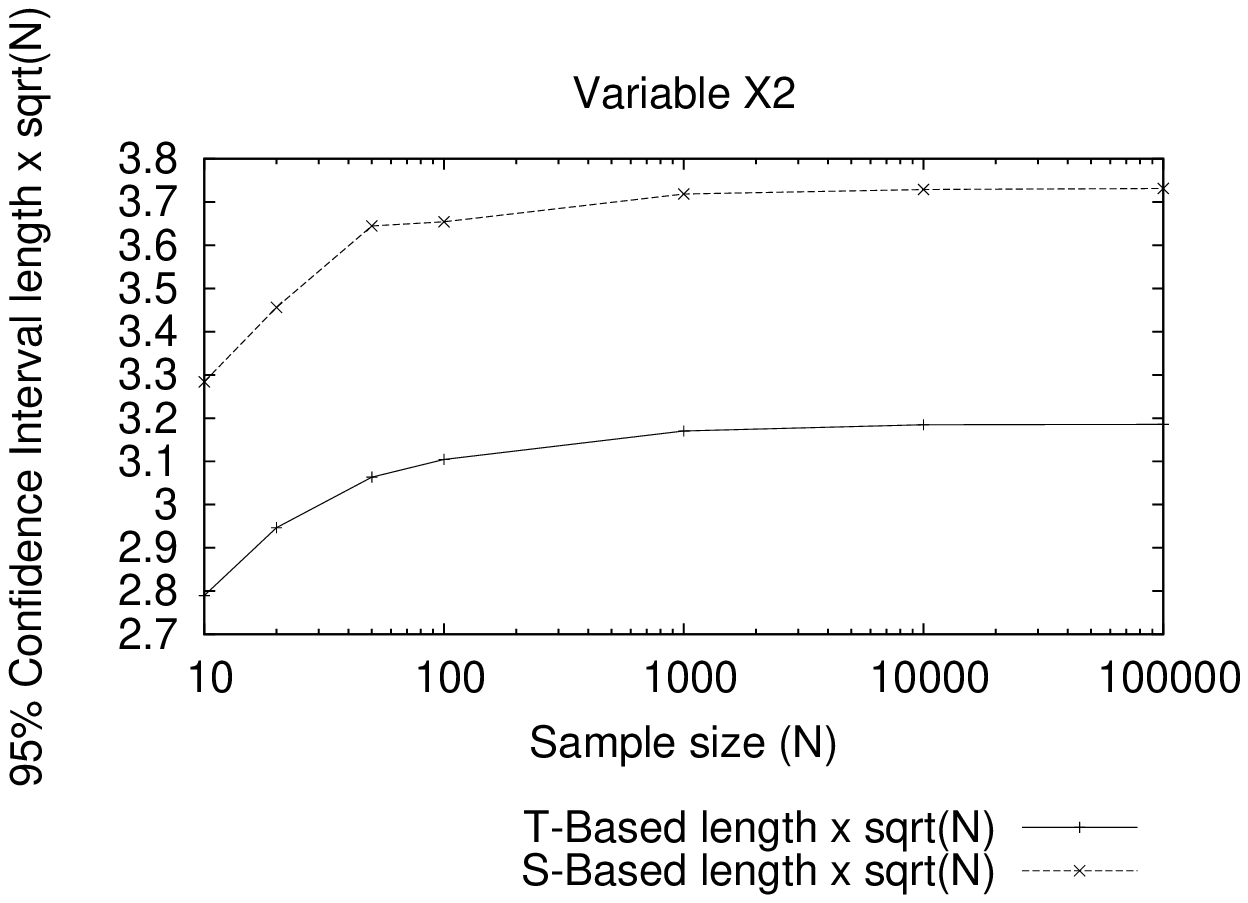}
		\includegraphics[scale=.4]{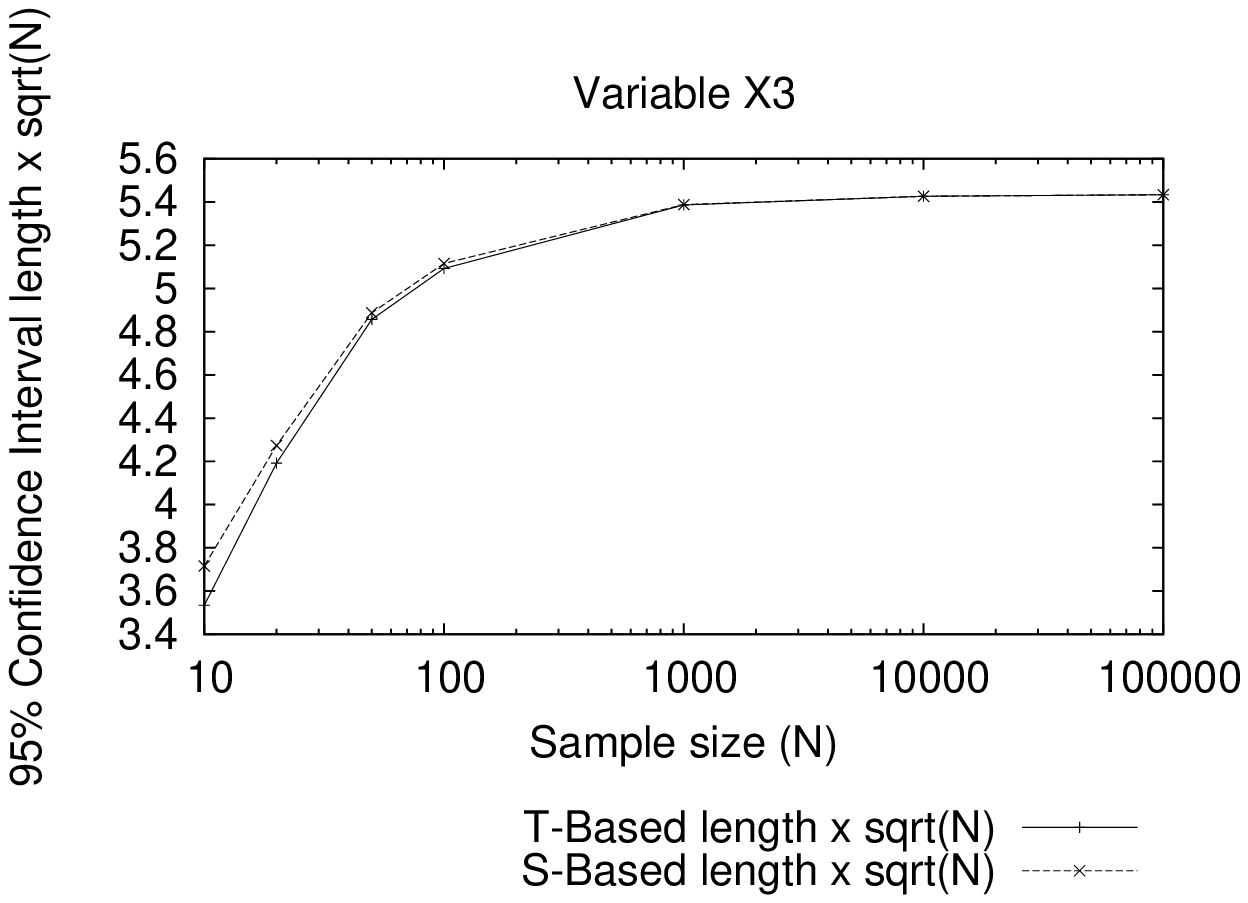}
		\caption{Lengths (rescaled by $\sqrt N$) of the estimated 95\% confidence intervals for $S^1$ (left), $S^2$ (center) and $S^3$ (right), as functions of the sample size (for the exact model). In solid line: length of the interval built from $T_N$ estimator; in dotted line: length of the interval built from $S_N$ estimator. }
		\label{f:compaEstim}
	 \end{center}
\end{figure}

We see that the coverages get closer to the target level $0.95$ as $N$ increases, thereby assessing the reliability of the asymptotic confidence interval.

Figure \ref{f:compaEstim} compares the efficiency of $S_N^X$ and $T_N^X$ by plotting the confidence interval lengths for the two estimators, as functions of the sample size. As the lengths for both estimators are $O(1/\sqrt N)$, we plot the lengths multiplied by $\sqrt N$. We see that $T_N^X$ always produce smaller confidence intervals, except for $X_3$ where the lengths are sensibly the same; this conclusion fully agrees with Proposition \ref{prop:varless}.

\subsection{Gaussian-perturbated model} \label{ss:res2}
We consider a perturbation $\widetilde f_N$ of the original output $f$:
\[ \widetilde f_N=f+\frac{5 \xi}{N^{\beta/2}} \]
where $\beta>0$ and $\xi$ is a standard Gaussian. 

The perturbation $\delta_N=5 \frac{\xi}{N^{\beta/2}}$ leads to $\Var\delta_N \propto N^{-\beta}$. Since:
\[ C_\delta = O\left(\Var(\delta_N)^{1/2}\right)=O\left( N^{-\beta/2} \right), \]
the proof of Theorem \ref{cor:as_norm} shows that $\widetilde S_N$ is asymptotically normal for $S$ if $\beta>1/2$. For indices relative to $X_1$ and $X_2$, this sufficient condition is also necessary, as $C_\delta$ is actually equivalent to $N^{-\beta/2}$. For $X_3$, we have $C_\delta=0$ so that $\widetilde S_N$ is asymptotically normal for $S$ for any positive $\beta$.

This is illustrated for $N=50000$ in Figure \ref{f:2}. We see that the empirical coverages of the confidence interval for $S^1$ and $S^2$ jump to $0.95$ near $\beta=1/2$, while, for $S^3$, this coverage is always close to $0.95$.

\begin{figure}
	\begin{center}
		\includegraphics[scale=.4]{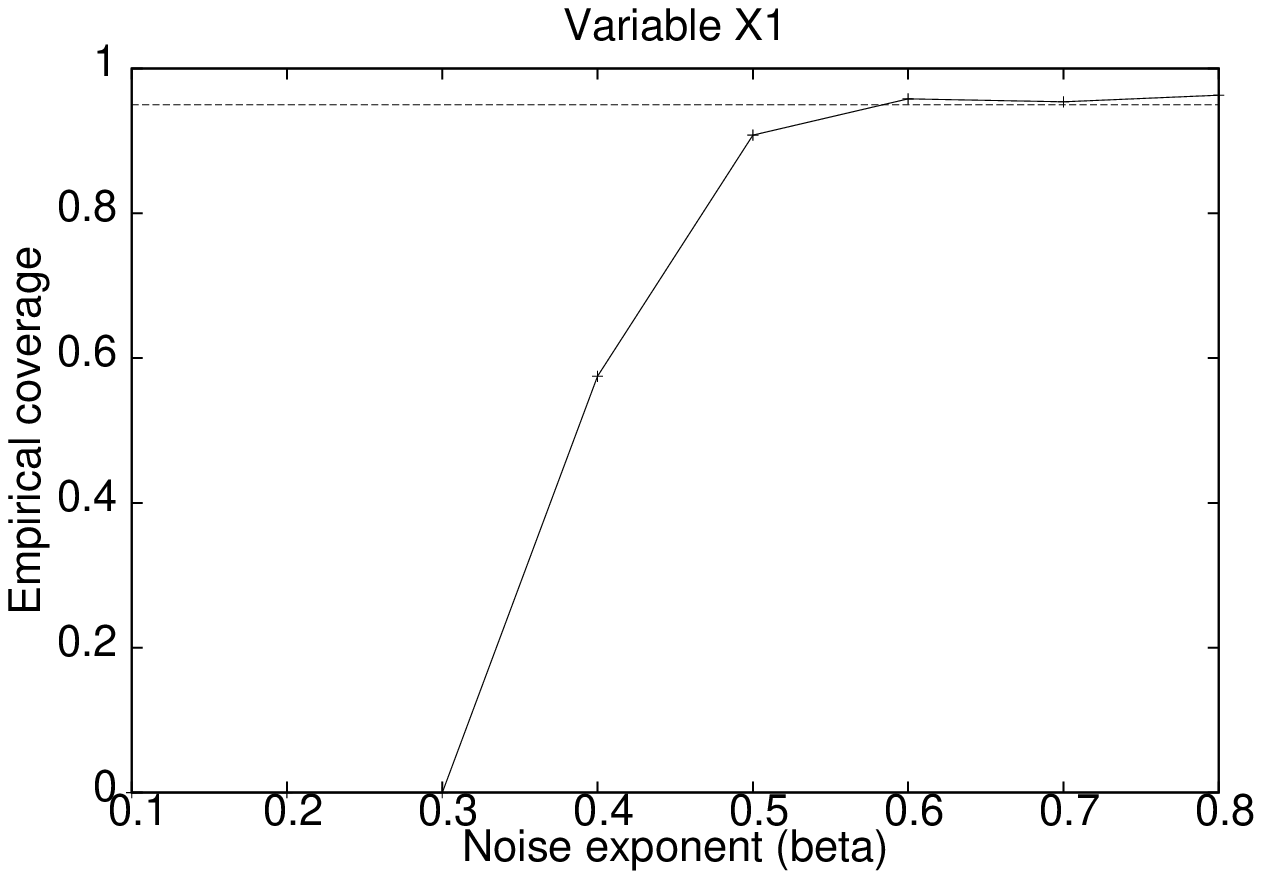}
		\includegraphics[scale=.4]{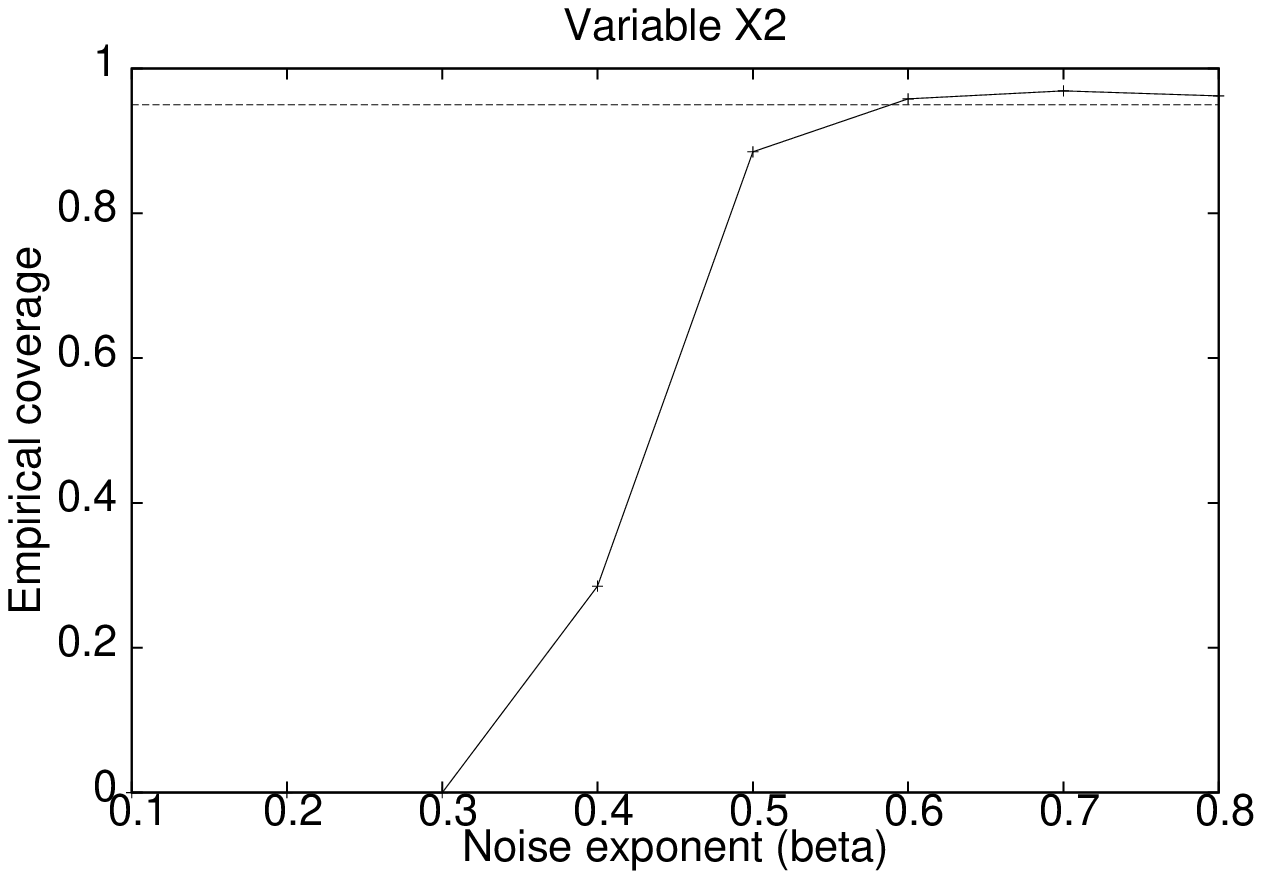}
		\includegraphics[scale=.4]{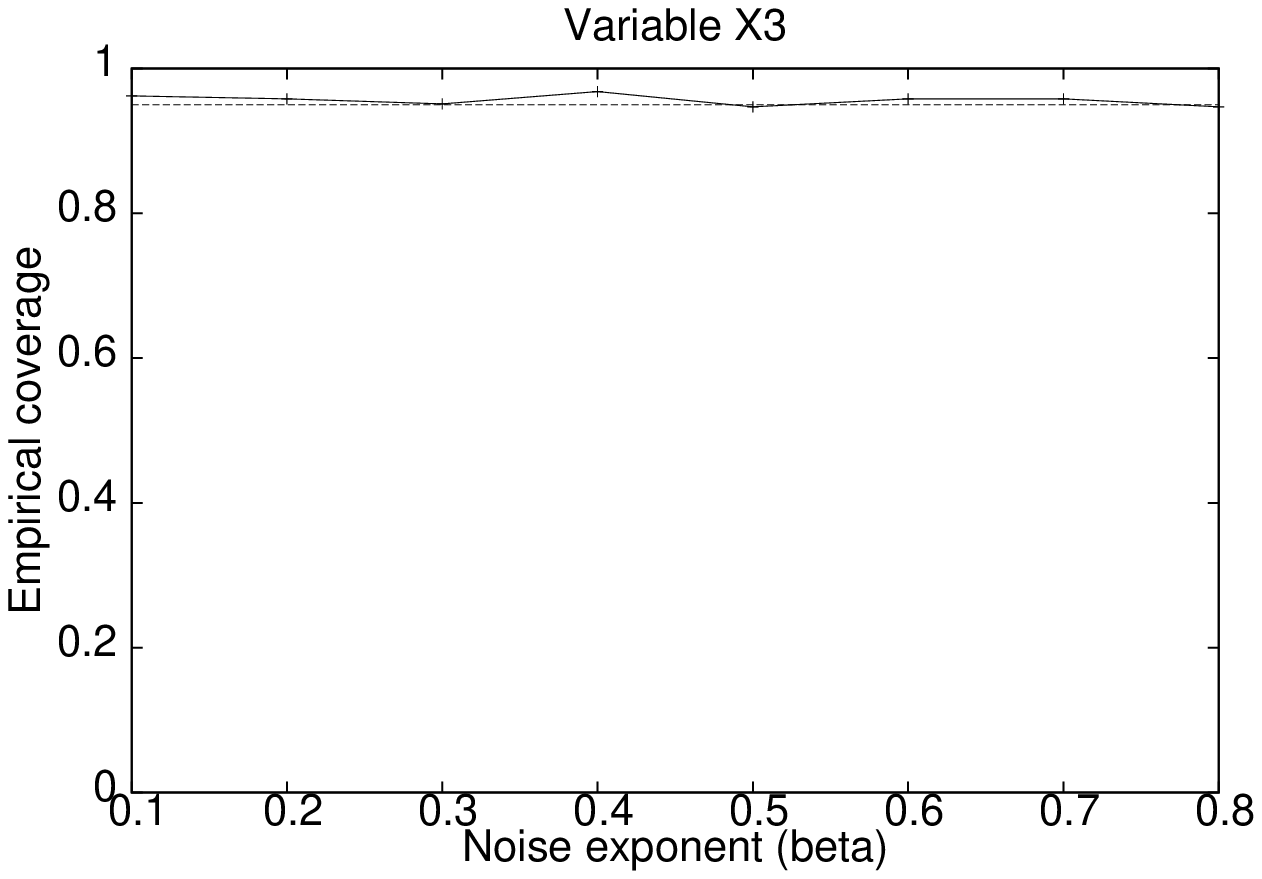}
		\caption{Empirical coverages of the asymptotic confidence intervals for $S^1$, $S^2$ and $S^3$, as a function of $\beta$ (for the Gaussian-perturbated model). }
		\label{f:2}
	\end{center}
\end{figure}

\subsection{Weibull-perturbated model} \label{ss:res3}
We now take a different perturbation of the output:
\[ \widetilde f_N=f+\frac{5 W X_3^2}{N^{\beta/2}} \]
where $W$ is Weibull-distributed with scale parameter $\lambda=1$ and shape parameter $k=1/2$. Here, the perturbation depends on the inputs and, as for every input variable, $C_{\delta,N}$ does not converge to zero, Theorem \ref{cor:as_norm} states in particular that $\widetilde S_N$ is asymptotically normal for $S$ for $\beta>1$. Again, this property is suggested for $N=50000$ by the plot in Figure \ref{f:3}.

\begin{figure}
	\begin{center}
		\includegraphics[scale=.4]{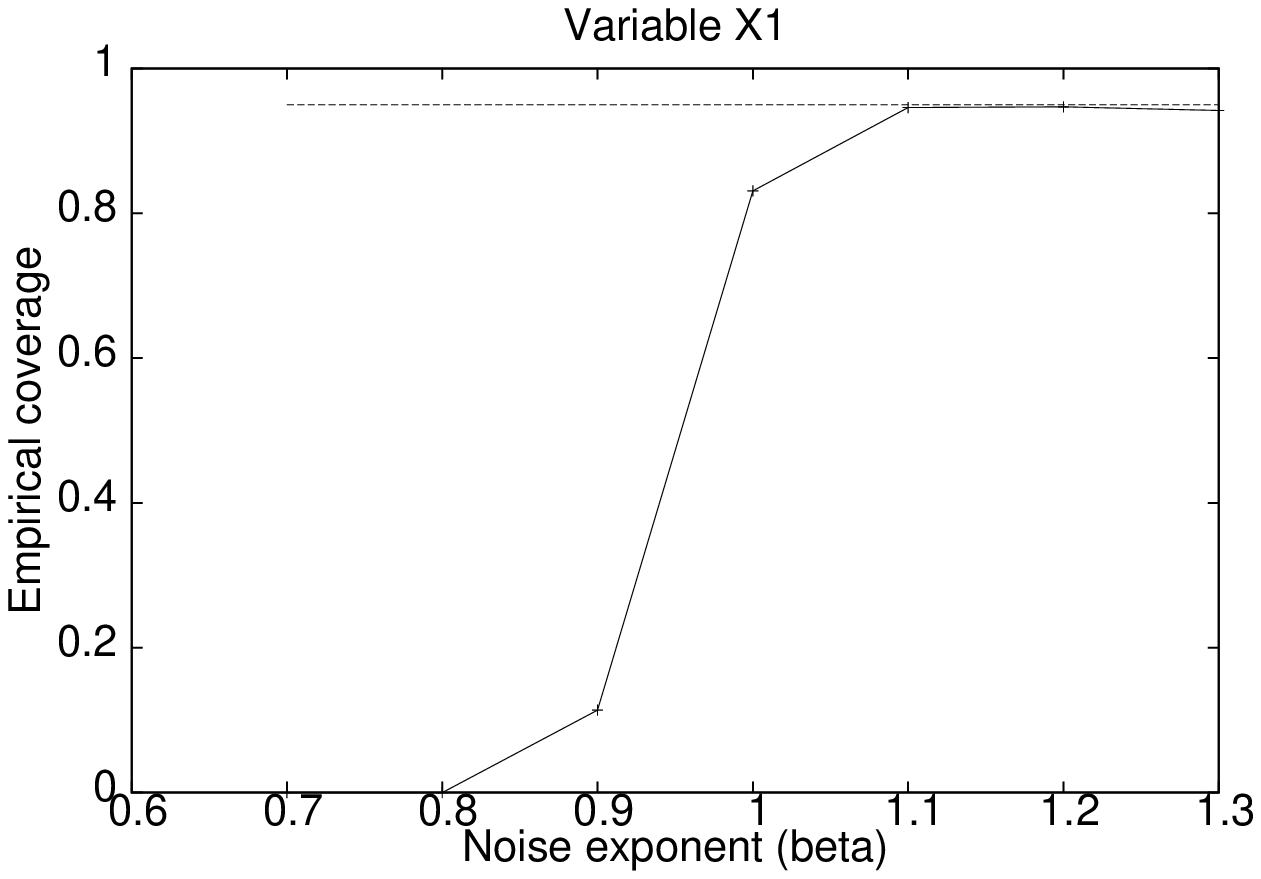}
		\includegraphics[scale=.4]{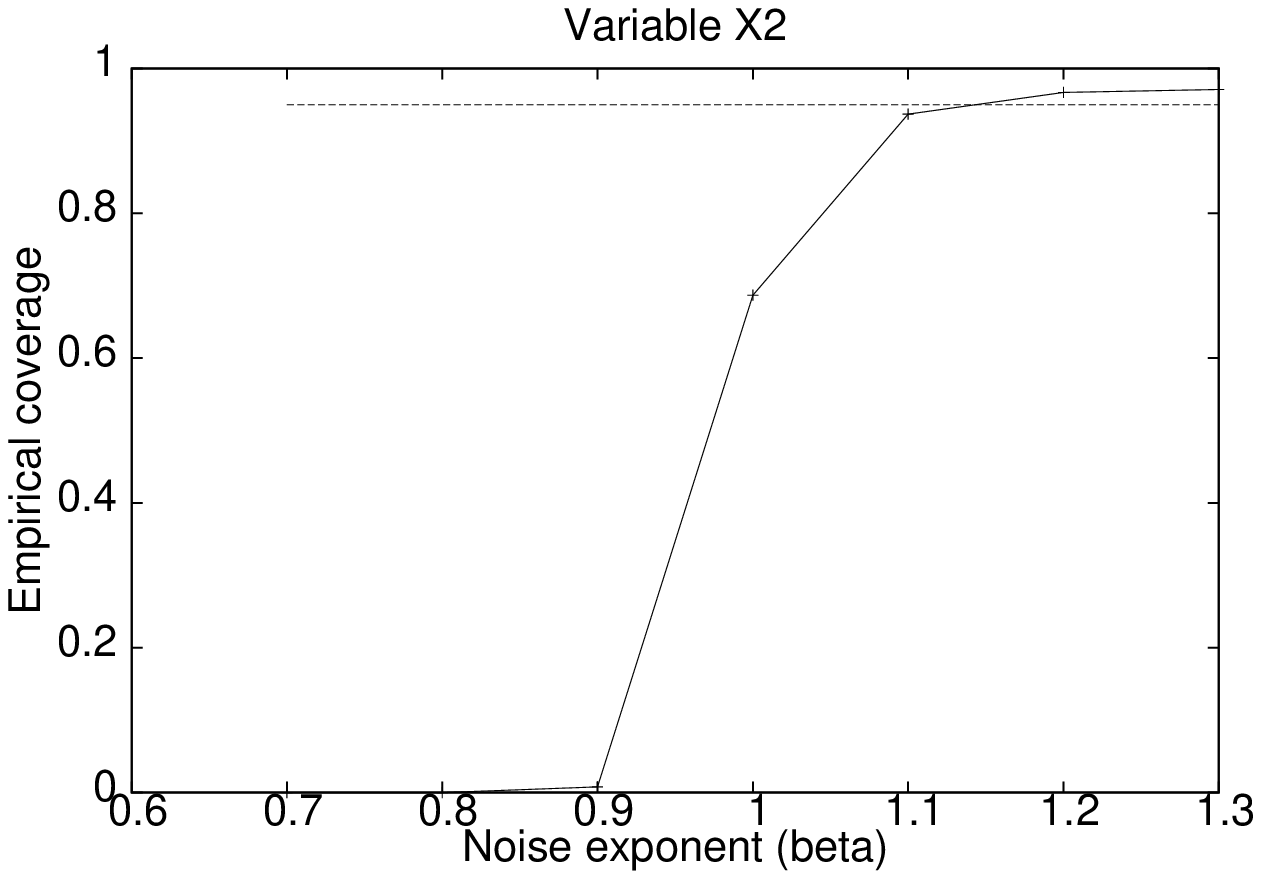}
		\includegraphics[scale=.4]{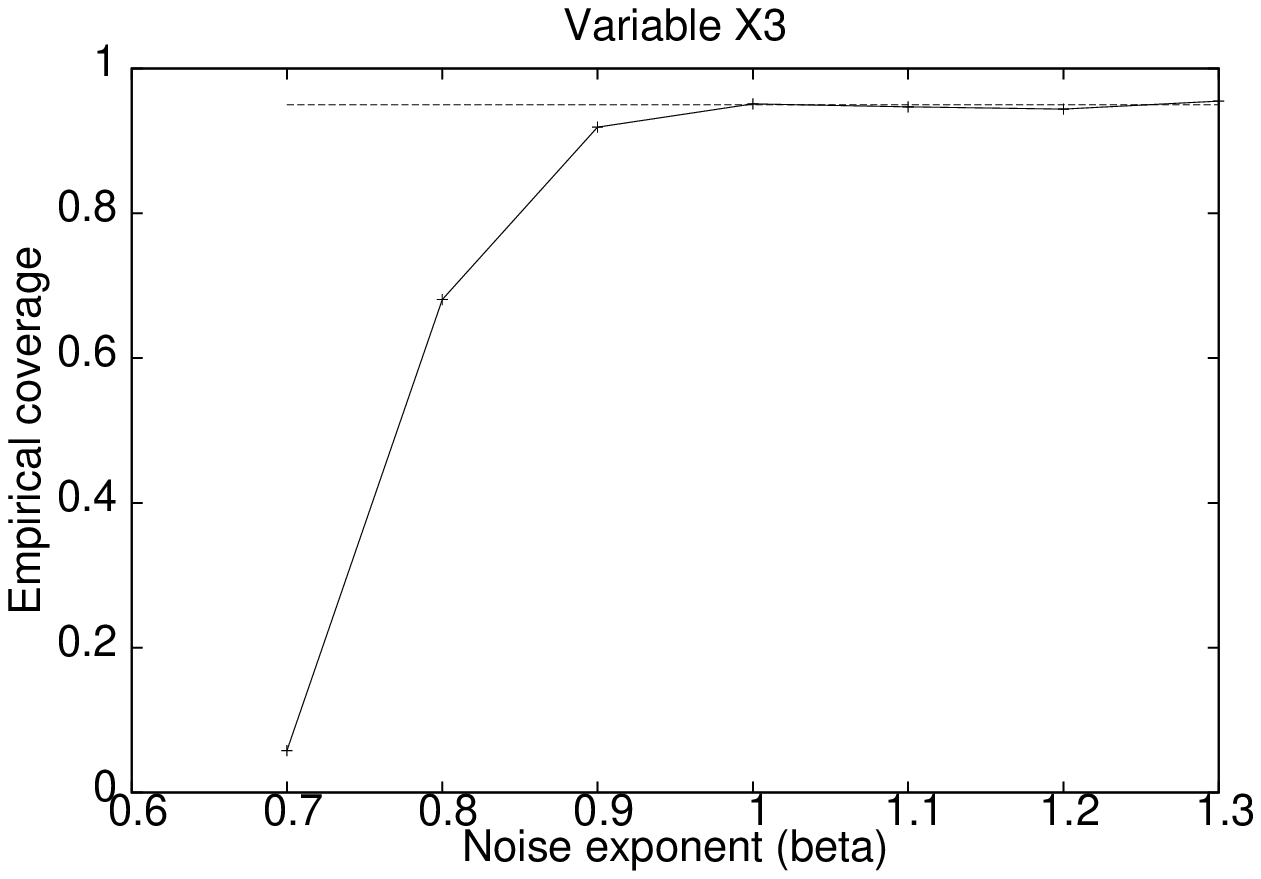}
		\caption{Empirical coverages of the asymptotic confidence intervals for $S^1$, $S^2$ and $S^3$, as a function of $\beta$ (for the Weibull-perturbated model). }
		\label{f:3}
	\end{center}
\end{figure}

\subsection{RKHS metamodel} \label{ss:rkhsmeta}
In this part, we discuss the use of a reproducing kernel Hilbert space (RKHS) interpolator \cite{sant:will:notz:2003,schaback2003mathematical,schaback2011interp} as metamodel $\widetilde f$. Such metamodels (also known as Kriging, or Gaussian process metamodels) are widely used when performing sensitivity analysis of time-expensive computer codes \cite{marrel2009calculations}. Note that, according to \cite{chen2005analytical}, analytical formulae are in some cases (e.g., uniform or gaussian distributions for the inputs) available for Sobol indices computation, avoiding the necessity to use a Monte-Carlo scheme. In this paper, we chose to perform Monte-Carlo estimation on an RKHS metamodel so as to illustrate our theoretical results. Moreover, the Monte-Carlo approach is more flexible and can be applied for complex inputs' distributions. The interpolator depends on a learning sample $\{ (d_1, f(d_1)), \ldots, (d_n, f(d_n)) \}$, where the design points $\mathcal D = \{d_i\}_{i=1,\ldots,n} \subset \mathcal P$ are generally chosen according to a space-filling design, for instance the so-called maximin LHS (latin hypercube sampling) designs. Increasing the learning sample size $n$ will increase the necessary number of evaluations of the true model $f$ (each evaluation being potentially very computationally demanding) to build the learning sample, but will also enhance the quality of the interpolation (i.e. reduce metamodel error).

The error analysis of the RKHS method \cite{schaback2003mathematical,madych1992bounds} shows that there exist positive constants $\mathcal C$ and $\mathcal K$, depending on $f$, so that:
\[ \forall u \in \mathcal P, \;\; \abs{ f(u) - \widetilde f(u) } \leq \mathcal C e^{-\mathcal K/h_{\mathcal D,\mathcal P}} \]
where:
\[ h_{\mathcal D,\mathcal P} = \sup_{u \in \mathcal P} \min_{d \in \mathcal D} \norm{d-u} \]
for a given norm $\norm{\cdot}$ on $\mathcal P$.

The quantity $h_{\mathcal D,\mathcal P}$ can be linked to the number of points $n^*(\epsilon)$ in an optimal covering of $\mathcal D$:
\[ n^*(\epsilon) = \min \{p \in \N^* \;|\; \exists (d_1,\ldots,d_p)\in\mathcal P \text{ s.t. } \forall u \in \mathcal P, \exists i \in \{1,\ldots,p\} \text{ satisfying } \norm{u-d_i}\leq \epsilon \}. \]
In other words, $n^*(\epsilon)$, known as the \emph{covering number of $\mathcal P$}, is the smallest size of a design $\mathcal D$ satisfying $h_{\mathcal D,\mathcal P} \leq \epsilon$.

It is known that, when $\mathcal P$ is a compact subset of $\R^p$ (in our context, $p=p_1+p_2$ is the number of input parameters), there exist constants $A$ and $B$ so that:
\[ A \epsilon^{-p} \leq n^*(\epsilon) \leq B \epsilon^{-p}. \]
Hence, assuming that an optimal design of size $n$ is chosen, we have, for a constant $B'$:
\[ h_{\mathcal D,\mathcal P} \leq B' n^{-1/p} \]
and we have the following pointwise metamodel error bound, for constants $C$ and $K'$:
\[ \forall u \in \mathcal P, \;\; \abs{ f(u) - \widetilde f(u) } \leq \mathcal C e^{-\mathcal K' n^{1/p}} \]
which obviously leads to an integrated error bound on the variance of the metamodel error:
\[ \Var \delta \leq C e^{-k n^{1/p}} \]
for suitable constants $C$ and $k$.

\subsubsection*{Numerical illustration}
We illustrate the properties of the RKHS-based sensitivity analysis using the Ishigami function \eqref{e:ishigam} as true model, maximin LHSes for design points selection. RKHS interpolation also depends on the choice of a kernel, which we choose Gaussian all the way through. All simulations have been made with the \verb<R< software \cite{R2011}, together with the \verb<lhs< package \cite{Rlhs} for design sampling and the \verb<mlegp< package \cite{Rmlegp} for Kriging.

Figure \ref{f:k1}, which shows an estimation (based on a sample of 1000 metamodel errors) of the (logarithm of) variance of metamodel error, plotted against the cubed root of the learning sample size $n^{1/3}$.
\begin{figure}
	\begin{center}
		\includegraphics[scale=.4,angle=270]{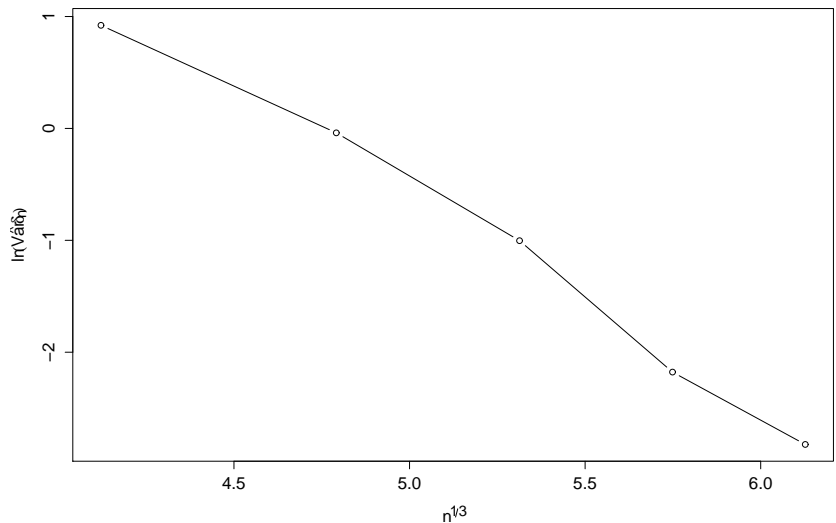}
		\caption{Estimation of the Kriging metamodel error variance (log. scale) as function of the learning sample size $n$. }
		\label{f:k1}
	\end{center}
\end{figure}
Using an exponential regression, we find that:
\begin{equation}\label{e:approxVarDelta} \Var(\delta)\approx \widehat C e^{-\widehat k n^{1/3}} \end{equation}
where:
\[ \widehat k = 1.91 \]
Now, if we let the learning sample size $n$ depend on the Monte-Carlo sample size $N$ by the relation:
\[ n=(a \ln N)^3 \]
for $a>0$, Theorem \ref{cor:as_norm} suggests that the metamodel-based estimators of the sensitivity indices are asymptotically normal if and only if $N^{-a\widehat k+1} \rightarrow 0$ when $N\rightarrow+\infty$, that is $a>\frac{1}{\widehat k}$, or 
        \begin{equation}
        \label{e:acrit}
        a > 0.52,
        \end{equation}according to our numerical value for $\widehat k$.

		  Even if it has not been rigorously proved that this condition is necessary and sufficient (due to the estimation of $k$ and the fact that \eqref{e:approxVarDelta} provably holds, possibly with different constants, as an upper bound), one should observe in practice that the behavior of the empirical confidence intervals for large values of $N$ changes as this critical value of $a$ is crossed.  Table \ref{tbl:1} below shows the results obtained for different subcritical and supercritical values of $a$ (i.e., \eqref{e:acrit} does not hold, or hold, respectively), and provides a clear illustration of this fact. 
\begin{table}
\begin{center}
\begin{tabular}{|c|c|c|c|c|c|}
\hline
$a$ & $N$ & $n$ & Coverage for $S^1$ & Cov. for $S^2$ & Cov. for $S^3$ \\ \hline
.4 & 3000  & 33  &   0.1              &       0        &       0.7       \\ \hline
.4 & 4000  &  37 &   0.08             &       0        &      0.78       \\ \hline
.4 & 6000  & 43  &   0.26             &       0.3      &      0.88       \\ \hline
.4 & 10000 &  51 &   0.28             &       0.18     &      0.78       \\ \hline
.4 & 20000 &  77 &   0.28             &       0.1     &       0.59       \\ \hline
.6 & 3000  & 111 &   0.79             &       0.37     &      0.9        \\ \hline
.6 & 4000  & 124 &   0.8              &       0.7      &      0.94       \\ \hline
.6 & 10000 & 169 &   0.92             &       0.82     &      0.94       \\ \hline
.6 & 20000 & 210 &   0.93             &       0.85     &      0.95       \\ \hline
.7 & 3000  & 177 &   0.93             &       0.88     &      0.93       \\ \hline
.7 & 4000  & 196 &   0.9              &       0.91     &      0.94       \\ \hline
.7 & 6000  & 226 &   0.94             &       0.93     &      0.97       \\ \hline
.8 & 4000  & 293 &   0.95             &       0.95     &      0.95       \\ \hline
\end{tabular}
\caption{Estimation of the asymptotic coverages for the RKHS Ishigami metamodel. Empirical coverages are obtained using 100 confidence interval replicates. Theoretical coverage is 0.95. }
\label{tbl:1}
\end{center}
\end{table}

\subsection{Nonparametric regression} \label{ss:regrmeta}
In this section, we consider the case where the true model $f$ is not directly observable, but is only available through a finite set of \emph{noisy} realisations of:
\[ f_{\text{noisy}}(D_i)=f(D_i)+\epsilon_i, \;\;\; i=1,\ldots,n \]
where $\mathcal D=\left( D_i=(X_i,Z_i) \right)_{i=1,\ldots,n}$ are independent copies of $(X,Z)$, and $\{ \epsilon_i \}_{i=1,\ldots,n}$ are independent, identically distributed centered random variables.

As discussed in Section \ref{s:stationnaire}, one should expect that the Sobol index estimator computed on $f_{\text{noisy}}$ are not asymptotically normal for the estimation of the Sobol indices of $f$ (as $\Var(\epsilon_i)$ is fixed). This motivates the use of a smoothed estimate of $f$, which we will take as our perturbated model $\widetilde f=\widetilde f_{\mathcal D}$. We consider the Nadaraya-Watson estimator:
\[ \widetilde f_{\mathcal D}(u) = \left\{ \begin{array}{l}
			\displaystyle\frac{ \sum_{i=1}^n K_h \left( u-D_i \right) f_{\text{noisy}}(D_i) }{ \sum_{i=1}^n K_h(u-D_i) } \text{ if } \sum_{i=1}^n K_h(u-D_i) \neq 0 \\
			0 \text{ else. } \end{array} \right. \]
where $K_h$ is a smoothing kernel of window $h \in \R^p$; for instance $K_h$ is a Gaussian kernel:
\begin{equation}\label{e:gausskern} K_h(v)=\exp\left( - \sum_{i=1}^{p} \frac{ \norm{v_i}^2 }{ h_i^2 } \right) \end{equation}
where the norm $\norm{\cdot}$ is the Euclidean norm on $\R^p$.

It is known that, under regularity conditions on $f$, and a $n$-dependent appropriate choice of $h$, the mean integrated square error (MISE) of $\widetilde f$ satisfies:
\begin{equation}\label{e:majoint} \int  \E_{\mathcal D}\left(\left( f(u) - \widetilde f_{\mathcal D}(u) \right)^2 \right) \ud u \leq C' n^{-\gamma}, \end{equation}
for a positive constant $C'$ and a positive $\gamma$ (which depends only on the dimension $p$ and the regularity of $f$), and where $\E_{\mathcal D}$ denotes expectation with respect to the random ``design'' $\mathcal D$.

Now, by Fubini-Tonelli's theorem, we have:
\begin{equation} \label{e:fubton} \int  \E_{\mathcal D}\left(\left( f(u) - \widetilde f_{\mathcal D}(u) \right)^2 \right) \ud u = \E_{\mathcal D}\left( \int \left( f(u) - \widetilde f_{\mathcal D}(u) \right)^2  \ud u \right). \end{equation}
By using \eqref{e:fubton}, \eqref{e:majoint} and applying Markov's inequality to the positive random variable $\displaystyle\int \left( f(u)-\widetilde f_{\mathcal D}(u) \right)^2 \ud u$, we have that, for any $\epsilon>0$,
\[ \P\left( \left\{ \mathcal D \;/\; \int \left( f(u)-\widetilde f_{\mathcal D}(u) \right)^2 \ud u \leq \frac{C'}{\epsilon}n^{-\gamma} \right\} \right) \geq 1 - \epsilon. \]
Hence, for a fixed risk $\epsilon>0$, there exist $C>0$ and $\gamma>0$ so that:
\begin{equation}\label{e:majproba} \int \left( \widetilde f_{\mathcal D}(u) - f(u) \right)^2 \ud u \leq C n^{-\gamma} \end{equation}
holds with probability greater than $1-\epsilon$ (with respect to the choice of $\mathcal D$).

We recall that the quantity we have to consider in order to study asymptotic normality of Sobol index estimator on the metamodel is:
\[ \Var(\delta) = \int \left( f(u)-\widetilde f_{\mathcal D}(u) \right)^2 \ud u - \left( \int \left(f(u)-\widetilde f_{\mathcal D}(u)\right) \ud u \right)^2 \]
and that, obviously,
\[ \Var(\delta) \leq  \int \left( f(u)-\widetilde f_{\mathcal D}(u) \right)^2 \ud u. \]
This gives, by making use of \eqref{e:majproba}:
\begin{equation} \label{e:majproba2} \Var\left(\delta\right) \leq C n^{-\gamma} \end{equation}
with probability greater than $1-\epsilon$.

In most cases of application, the design $\mathcal D$ is fixed. In view of \eqref{e:majproba2}, it is reasonable to suppose that there exist $C>0$ and $\beta>0$ so that:
\[ \Var\left(\delta\right) \leq C n^{-\beta} \]
and we make $n$ depend on $N$ by the following relation:
\[ n=N^{a}, \]
for $a>0$. By Theorem \ref{cor:as_norm}, the estimator sequence $\{ \widetilde S_N \}$ is asymptotically normal provided that $N \Var(\delta_N) \rightarrow 0$, that is: $a>\frac{1}{\beta}$.

\subsubsection*{Numerical illustration}
We now illustrate this property using the Ishigami function  \eqref{e:ishigam} as true model, and a Gaussian white noise $\epsilon_i$ of standard deviation $0.3$ (yielding to a signal-to-noise ratio of 90\%).

The nonparametric regressions are carried using a Gaussian kernel \eqref{e:gausskern}, the R package \verb<np< \cite{Rnp2008}, together with the extrapolation method of \cite{racine1993efficient} for window selection and the \verb<FIGtree< \cite{morariu08figtree} C++ library for efficient Nadaraya-Watson evaluation based on fast gaussian transform.

Figure \ref{f:benchNW}, which shows an estimation (based on a test sample of size 3000) of $\Var(\delta)$ in function of $n$, and a power regression shows that:
\[ \Var\left(\delta\right) \approx C n^{-\widehat\beta} \]
with $\widehat\beta=0.86$. This gives an estimate of $1.16$ as the critical $a$ for asymptotic normality.

\begin{figure}
	\begin{center}
		\includegraphics[scale=.4,angle=270]{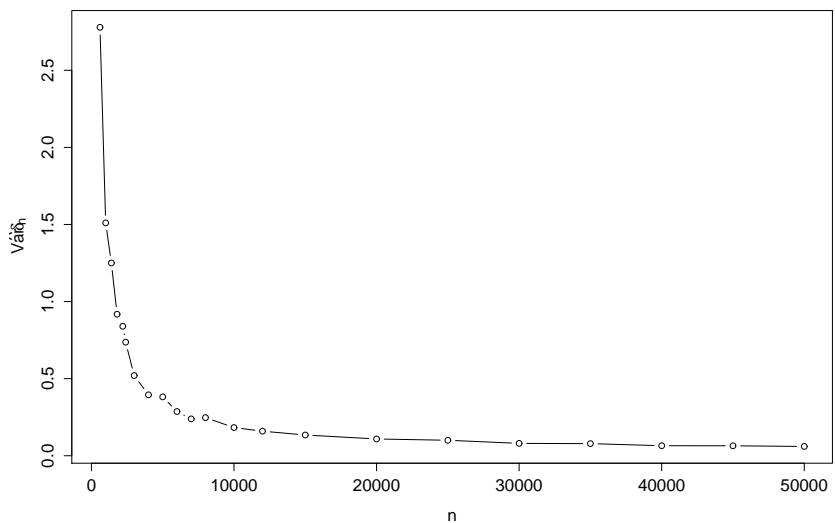}
		\caption{Estimation of the nonparametric regression error variance (log. scale) as function of the learning sample size $n$ (Subsection \ref{ss:regrmeta}). }
		\label{f:benchNW}
	\end{center}
\end{figure}

\begin{table}
\begin{center}
\begin{tabular}{|c|c|c|c|c|c|}
\hline
$a$ & $N$ & $n$ & Coverage for $S^1$ & Cov. for $S^2$ & Cov. for $S^3$ \\ \hline
0.8 & 1000 & 252 & 0.25 & 0.01 & 0.94 \\ \hline
0.8 & 2000 & 438 & 0.05 & 0.02 & 0.86  \\ \hline
1.1 & 1000 & 1996 & 0.95 & 0.97 & 0.96 \\ \hline
1.1 & 2000 & 4277 & 0.95 & 0.93 & 0.96 \\ \hline
1.2 & 1000 & 3982 & 0.93 & 0.95 & 0.96 \\ \hline
1.2 & 2000 & 9147 & 0.96 & 0.97 & 0.95 \\ \hline
1.3 & 1000 & 7944 & 0.95 & 0.99 & 0.94 \\ \hline
1.3 & 2000 & 19559 & 0.95 & 0.95 & 0.96  \\ \hline
\end{tabular}
\caption{Estimation of the asymptotic coverages for the Ishigami nonparametric regression. Empirical coverages are obtained using 100 confidence interval replicates. Theoretical coverage is 0.95. }
\label{tbl:2}
\end{center}
\end{table}

As in the RKHS case, we performed estimations of the coverages of the asymptotic confidence interval for several values of $a$ and $N$; the results are gathered in Table \ref{tbl:2}. We see that, first, the condition $a>1.16$ implies correct coverages, and, second, the condition also seems to be near-necessary to have asymptotic normality. We also remark that, for the asymptotic normality to hold, the necessary number of noisy model evaluations is asymptotically comparable to the Monte-Carlo sample size (while, in the RKHS case, the necessary number of true model evaluations was asymptotically negligible with respect to the Monte-Carlo sample size): this shows that the nonparametric regression is suitable in the case of noisy but abundant model evaluations, while RKHS interpolation is clearly preferable when the true model output is costly to evaluate (i.e. few model outputs are available).

\section{Appendix: Proofs}
\begin{proof}[Proof of Lemma \ref{lemma:cov}]
	On one hand, since $Y \buildrel \mathcal{L} \over = Y^X$ (that is, $Y$ and $Y^X$ have the same distribution), we have 
$$\Cov(Y,Y^X)=\mathbb{E}(YY^X)-\mathbb{E}(Y)\mathbb{E}(Y^X)=\mathbb{E}(YY^X)-\mathbb{E}(Y)^2.$$
On the other hand, $Y$ and $Y^X$ are independent conditionally on $X$, so that
\begin{equation*}
\mathbb{E}(YY^X)=\mathbb{E}(\mathbb{E}(YY^X|X))
= \mathbb{E}(\mathbb{E}(Y|X)\mathbb{E}(Y^X|X))
= \mathbb{E}(\mathbb{E}(Y|X)^2).
\end{equation*}
\end{proof}

\begin{proof}[Proposition \ref{prop:an1}: Proof of \eqref{cv_n_a}] 
	We begin by noticing that $S_N^X$ is invariant by any centering (translation) of the $Y_i$ and $Y_i^X$. To simplify the next calculations, we suppose that they have been recentred by $-\E(Y)$.
	By setting:
	\begin{equation}\label{e:defui} U_i = \left( (Y_i - \E(Y))(Y_i^X - \E(Y)), \quad
			                     Y_i - \E(Y), \quad
										Y_i^X - \E(Y), \quad
										(Y_i-\E(Y))^2 \right)^T, \end{equation}
	this implies that:
	\[ S_N^X = \Psi_S( \overline U_N ) \]
	with:
	\[ \Psi_S (x,y,z,t) = \frac{ x-yz }{ t - y^2 } \]
	The central limit theorem gives that:
	\[ \sqrt N \left( \overline U_N - \mu \right)
       \cvloi
		 \mathcal N_4 \left( 0, \Gamma \right) \]
	 where $\Gamma$ is the covariance matrix of $U_1$ and:
	 \[ \mu=  \begin{pmatrix} \Cov(Y,Y^X) \\
				                                                 0 \\
																				 0 \\
																				 \Var(Y) \end{pmatrix}. \]
	 
	 The so-called Delta method \cite{van2000asymptotic} (Theorem 3.1) gives:
	 \[ \sqrt N \left( S_N^X - S^X \right) \cvloi \mathcal N_1(0, g^T \Gamma g) \]
	 where: 
	 \[ g =  \nabla \Psi_S (\mu).  \]
Note that since by assumption $\Var(Y) \neq 0$, $\Psi_S$ is differentiable at $\mu$ and we will see that $g^T \Gamma g \neq 0$, so that the application of the Delta method is justified. 
	 By differentiation, we get that, for any $x,y,z,t$ so that $t \neq y^2$:
	 \[ 	\nabla \Psi_S ( x, y, z, t) =
		 \left( \frac{1}{t-y^2}, \quad
			      \frac{ - z(t-y^2)+(x-yz)\cdot 2y }{ (t-y^2)^2 } , \quad
					- \frac{y}{t-y^2} , \quad
					- \frac{x-yz}{(t-y^2)^2} \right)^T \]
so that, by using \eqref{sobol_cov}:
   \[ g = \left( \frac{1}{\Var(Y)} , \quad
					                  0 , \quad
											0 , \quad
											- \frac{S^X}{\Var(Y)} \right)^T. \]
Hence 
\begin{eqnarray*}
	g^T \Gamma g &=& \frac{ \Var\left( (Y-\E(Y))(Y^X-\E(Y)) \right) }{ (\Var(Y))^2 }
	    + \frac{ (S^X)^2 }{ (\Var(Y))^2 } \Var \left( (Y-\E(Y))^2 \right) \\
		 &&- 2 \frac{ S^X }{ (\Var(Y))^2 } \Cov \left( (Y-\E(Y))(Y^X-\E(Y)), (Y-\E(Y))^2 \right) \\
		 &=& \frac{1}{(\Var(Y))^2} \Biggl( \Var\left( (Y-\E(Y))(Y^X-\E(Y)) \right) + \Var\left( S^X \left( (Y-\E(Y))^2 \right) \right)\\
			&& - 2 \Cov\left( (Y-\E(Y))(Y^X-\E(Y)), S^X (Y-\E(Y))^2 \right) \Biggr) \\
		 &=& \frac{ \Var\left( (Y-\E(Y))\bigl[ (Y^X-\E(Y))-S^X (Y-\E(Y)) \bigr] \right) }{(\Var(Y))^2},
\end{eqnarray*}
which is the announced result.

	\emph{Proof of \eqref{cv_n_a_2}. } As in the previous point, it is easy to check that $T_N^X$ is invariant with respect to translations of $Y_i$ and $Y_i^X$ by $-\E(Y)$. Thus, $T_N^X=\Psi\left( \overline W_N \right)$ with:
\[ \Psi_T(x,y,z)=\frac{ x-(y/2)^2 }{ z/2 - (y/2)^2 } \]
and:
\begin{equation}\label{e:defwi}  W_i = \left( (Y_i-\E(Y))(Y_i^X-\E(Y)), \quad (Y_i-\E(Y))+(Y_i^X-\E(Y)), \quad
		 (Y_i-\E(Y))^2 + (Y_i^X-\E(Y))^2 \right)^T. \end{equation}
The result follows from the delta method.

\end{proof}

\begin{proof}[Proof of Proposition \ref{prop:varless}]
	We have that the expressions in \eqref{cv_n_a} and \eqref{cv_n_a_2} of $\sigma_S^2$ and $\sigma_T^2$ 	are translation-invariant, so that we assume without loss of generality that $\E(Y)=0$. By expanding the variances and using the exchangeability of $Y$ and $Y^X$, we have that $(\Var(Y))^2  \left(\sigma_S^2 - \sigma_T^2\right)$ is equal to:
	\begin{equation*} 
	(\Var(Y))^2  \left(\sigma_S^2 - \sigma_T^2\right)  
	= \frac{(S^X)^2}{2} \left( \Var(Y^2) - \Cov\left(Y^2, (Y^X)^2\right) \right).
\end{equation*}

We now use Cauchy-Schwarz inequality to see that:
\[ \Cov \left( Y^2, (Y^X)^2 \right) \leq \sqrt{ \Var\left(Y^2\right) \Var\left (Y^X)^2\right) } = \Var\left(Y^2\right) \]
so the second term is always non-negative. This proves that the asymptotic variance of $S_N^X$ is greater than the asymptotic variance of $T_N^X$.

For the equality case, we notice that $S^X=0$ implies the equality of the asymptotic variances. If $S^X \neq 0$, equality holds if and only if there is equality in Cauchy-Schwarz, ie. there exists $k \in \R$ so that:
\[ Y^2 = k (Y^X)^2 \text{ almost surely} \]
by taking expectations and using $\Var(Y)=\Var(Y^X)$ we see that $k=1$ necessarily, hence $Y=Y^X$ almost surely, and $S^X=1$ thanks to \eqref{sobol_cov}. 
\end{proof}

\begin{proof}[Proof of Lemma \ref{lemm:eff}]
Let, for $g  \in L^2(P)$ and $t\in\R$, $P_t^g$ be the cdf satisfying:
\[ \ud P_t^g = (1+t g) \ud P. \]
It is clear that the tangent set of $\mathcal P$ at $P$ is the closure of:
\[ \dot{\mathcal P_P} = \{ g 
\text{  bounded, } \E(g(Y,Y^X))=0 \text{ and } g(a,b)=g(b,a)\; \forall (a,b)\in\R^2 \}. \]

Let, for $Q \in \mathcal P$:
\[ \Psi_1(Q) = \E_Q \left(\Phi_1(Y)\right) \quad \text{ and } \quad \Psi_2(Q) = \E_Q \left( \Phi_2(Y,Y^X) \right). \]
We recall that $\E_Q$ denotes the expectation obtained by assuming that the random vector $(Y,Y^X)$ follows the $Q$ distribution.

Following \cite{van2000asymptotic} Section 25.3, we compute the efficient influence functions of $\Psi_1$ and $\Psi_2$ with respect to $\mathcal P$ and the tangent set $\dot{\mathcal P_P}$. These empirical influence functions are related to the minimal asymptotic variance of a regular estimator sequence whose observations lie in $\mathcal P$ (op.cit., Theorems 25.20 and 25.21).  Let $g \in \dot{\mathcal P_P}$. 

\begin{enumerate}
\item We have
\begin{eqnarray*}
	\frac{\Psi_1(P_t^g)-\Psi_1(P)}{t} &=& \E_P \left( \Phi_1(Y) g(Y,Y^X) \right) \\
	&=& \E_P \left[ \left( \frac{\Phi_1(Y)+\Phi_1(Y^X)}{2} - \E(\Phi_1(Y)) \right)  g(Y,Y^X). \right] 
\end{eqnarray*}
As:
\[ \widetilde {\Psi_{1,P}} = \frac{\Phi_1(Y)+\Phi_1(Y^X)}{2} - \E(\Phi_1(Y)) \in \dot{\mathcal P_P}, \]
 it is the efficient influence function of $\Psi_1$ at $P$. Hence the efficient asymptotic variance is:
 \[ \E_P\left(\left(\widetilde{\Psi_{1,P}}\right)^2\right) = \frac{ \Var\left( \Phi_1(Y)+\Phi_1(Y^X) \right) }{ 4 }. \]
As, by the central limit theorem, $\left\{ \Phi_N^1 \right\}$ clearly achieves this efficient asymptotic variance, it is an asymptotically efficient estimator of $\Psi_1(P)$.

\item We have:
\begin{eqnarray*}
	\frac{\Psi_2(P_t^g)-\Psi_2(P)}{t} &=& \E_P \left( \Phi_2(Y,Y^X) g(Y,Y^X) \right) \\
	&=& \E_P \left[ \left( \Phi_2(Y,Y^X)-\E(\Phi_2(Y,Y^X)) \right) g(Y,Y^X) \right].
\end{eqnarray*}
Thanks to the symmetry of $\Phi_2$, we have that 
\[ \widetilde{\Psi_{2,P}} =  \Phi_2(Y,Y^X)-\E(\Phi_2(Y,Y^X)) \]
belongs to $\dot{\mathcal P_P}$, hence it is the efficient influence function of $\Psi_2$. So the efficient asymptotic variance is:
\[ \E_P\left(\left(\widetilde{\Psi_{2,P}}\right)^2\right) = \Var\left( \Phi_2(Y,Y^X) \right), \]
and this variance is achieved by $\left\{ \Phi_N^2 \right\}$.\hfill \qedhere
\end{enumerate}
\end{proof}

\begin{proof}[Proof of Proposition \ref{prop:ae}]
By Lemma \ref{lemm:eff}, we get that:
\begin{equation}\label{e:defun}
 U_N = \left( \frac{1}{N} \sum_{i=1}^N Y_i Y_i^X, \quad
  \frac{1}{N} \sum_{i=1}^N \frac{Y_i + Y_i^X}{2}, \quad
  \frac{1}{N} \sum_{i=1}^N \frac{Y_i^2+(Y_i^X)^2}{2} \right) \end{equation}
is asymptotically efficient, componentwise, for estimating 
\begin{equation}\label{e:defu} U = \left( \E(YY^X), \quad \E(Y), \quad \E(Y^2) \right) \end{equation}
in $\mathcal P$.

Using Theorem 25.50 (efficiency in product space) of \cite{van2000asymptotic}, we can deduce joint efficiency from this componentwise efficiency.

Now, let $\Psi$ be the function defined by:
\[ \Psi(x,y,z)=\frac{x-y^2}{z-y^2} \]
and $\Psi$ is differentiable on:
\[ \R^3 \setminus \left\{(x,y,z)\, \big|\, z \neq y^2 \right\}, \]
Theorem 25.47 (efficiency and Delta method) of \cite{van2000asymptotic} implies that $\left\{ \Psi\left( U_N\right) \right \}$ is asymptotically efficient for estimating $\Psi(U)$ for $P\in\mathcal P$. The conclusion follows, as $\Psi(U_N)=T_N^X$ and $\Psi(U)=S^X$.
\end{proof}

\begin{proof}[Proof of Proposition \ref{p:propAA}]
We clearly have that $\widetilde Y_N \overset{L^2}{\underset{N\rightarrow+\infty}{\longrightarrow}} Y+c$.

We deduce that:
\[ \Var\left( \widetilde Y_N \right) \underset{N\rightarrow+\infty}{\longrightarrow} \Var(Y+c)=\Var(Y) \]
and
\[ \E(\widetilde Y_N|Z) \underset{N\rightarrow+\infty}{\longrightarrow} \E(Y|Z)+c \text{ in } L^2. \]
From this last convergence we get \[ \Var \left(\E(\widetilde Y_N|Z)\right) \underset{N\rightarrow+\infty}{\longrightarrow} \Var\left(\E(Y|Z)\right) .\] 
	This proves that $\widetilde S^X=\Var\left(\E(\widetilde Y_N|Z)\right)/\Var\left(\widetilde Y_N\right)$ converges to $S^X=\Var\left(\E(Y|Z)\right)/\Var(Y)$ when $N$ goes to $+\infty$.
\end{proof}

\begin{proof}[Proof of Proposition \ref{prop:as_norm}]
\textbf{Proof of \eqref{cv_n_aa}. }
Let
\[ \widetilde{U}_{N,i}=\left( (\widetilde Y_{N,i}-\E(Y)) (\widetilde Y_{N,i}^X-\E(Y)), \widetilde Y_{N,i}-\E(Y), \widetilde Y_{N,i}^X-\E(Y), \left( \widetilde Y_{N,i}-\E(Y) \right)^2 \right) \]
and
\[\overline{\widetilde{U}}_{N}: = \frac{1}{N}\sum_{i=1}^N \widetilde{U}_{N,i}.\]

Using the Lindeberg-Feller central limit theorem (see e.g. \cite{van2000asymptotic} 2.27, with $Y_{N,i}=\widetilde{U}_{N,i}/\sqrt N$), we get:
\[ \sqrt N \left( \overline{\widetilde{U}}_{N} - \E\left(\widetilde{U}_{N,1}\right) \right)
\cvloi
 \mathcal N_4(0, \Gamma) \]
where $\Gamma$ is the covariance matrix of the $U_1$ vector defined in \eqref{e:defui}.

The use of this central limit theorem is justified by the fact that, under assumption  \eqref{e:ass4ps} of uniform boundedness of moments of $\widetilde Y_N$, there are $s'>0$ and $C'$ such that:
\[ \forall N,\;\; \E(||U_{N,i}||^{2+s'}) < C' \]
where $\norm{\cdot}$ is the standard Euclidean norm.

This ensures
\[ \forall \epsilon>0, \;\; \E(||\tilde U_{N,i}||^2 \mathbf{1}_{||\tilde U_{N,i}||>\epsilon\sqrt N}) \rightarrow 0. \]
Then
\[ \E( ||\tilde U_{N,i}||^2 \mathbf{1}_{||\tilde U_{N,i}||>\epsilon\sqrt N} ) =
	\E \left( \frac{||\tilde U_{N,i}||^{2+s'}}{||\tilde U_{N,i}||^{s'}} \mathbf{1}_{||\tilde U_{N,i}||>\epsilon\sqrt N} \right) \leq \frac{C'}{\epsilon^{s'}N^{s'/2} }. \]

This shows that for each $i$, $\left\{ \norm{\tilde U_{N,i}}^2 \right\}_N$ is uniformly integrable, hence, the variance-covariance matrix of $\tilde U_{N,i}$ converges to $\Gamma$ when $N\rightarrow+\infty$. As $\widetilde U_{N,i} \overset{\mathbb P}{\underset{N\rightarrow+\infty}{\longrightarrow}} U_i$, the same convergence holds in $L^2$ and the covariance matrices of $\widetilde U_{N,i}$ converge (as $N \rightarrow +\infty$) to $\Gamma$, the covariance matrix of $U_i$.

We conclude the proof by applying the Delta method as for the exact model (cf. the proof of Proposition \ref{prop:an1}). 

\textbf{Proof of \eqref{cv_n_aa2}. } We set:
\[ \widetilde{W}_{N,i} = \left( (\widetilde Y_i-\E(Y))(\widetilde Y_i^X-\E(Y)), \quad (\widetilde Y_i-\E(Y))+(\widetilde Y_i^X-\E(Y)), \quad
		 (\widetilde Y_i-\E(Y))^2 + (\widetilde Y_i^X-\E(Y))^2 \right)^T. \]
As in the previous point, the Lindeberg-Feller theorem can be applied to $\left\{\widetilde W_{N,i}\right\}$ to yield the convergence:
\[ \sqrt N \left( \overline{\widetilde{W}}_{N} - \E\left(\widetilde{W}_{1,1}\right) \right)
\cvloi
 \mathcal N_3(0, \Sigma) \]
where $\Sigma$ is the covariance matrix of $W_1$ defined in \eqref{e:defwi}. The conclusion follows again by an application of the Delta method as in the proof of Proposition \ref{prop:an1}.
\end{proof}

\begin{proof}[Proof of Theorem \ref{cor:as_norm}]
The following decompositions:
\begin{equation}\label{eq:decomp}
\sqrt{N}(\widetilde S_N^X-S^X) = \sqrt N (\widetilde S_N^X - \widetilde S^X) + \sqrt N (\widetilde S^X - S^X )
\end{equation}
\begin{equation}\label{eq:decomp2}
\sqrt{N}(\widetilde T_N^X-S^X) = \sqrt N (\widetilde T_N^X - \widetilde S^X) + \sqrt N (\widetilde S^X - S^X )
\end{equation}
make obvious that if $\sqrt N (\widetilde S^X - S^X)$ goes to some constant $\kappa$ then
\[ \sqrt{N}(\widetilde S_N- S) \overset{\mathcal L}{\underset{N\rightarrow+\infty}{\longrightarrow}} \mathcal N (\kappa, \sigma_S^2)\]
and:
\[ \sqrt{N}(\widetilde T_N- S) \overset{\mathcal L}{\underset{N\rightarrow+\infty}{\longrightarrow}} \mathcal N (\kappa, \sigma_T^2).\]

The second point of the theorem is now clear from the proof of Proposition \ref{prop:pasconsistant}. 

The remaining of the theorem is an immediate consequence of Lemma \ref{lemma100} below.
\end{proof}

\begin{lemma}\label{lemma100}
	We have:
\[ \sqrt N \left(\widetilde {S}^X-S^X\right) =\frac{O \left( \left(N \Var(\delta_N)\right)^{1/2} \right)}{\Var(Y)+o(1)}. \]
\end{lemma}

\begin{proof}[Proof of \ref{lemma100}]
We have:
\begin{eqnarray*}
	\widetilde S^X-S^X &=& \frac{\Cov(\widetilde Y_N,\widetilde Y^X_N)}{\Var \widetilde Y_N} - \frac{\Cov(Y,Y^X)}{\Var(Y)} \\
	&=& \frac{\Cov(Y,Y^X)+2\Cov(Y,\delta^X_N)+\Cov(\delta_N,\delta^X_N)}{\Var(Y)+2\Cov(Y,\delta_N)+\Var(\delta_N)} - \frac{\Cov(Y,Y^X)}{\Var(Y)}\\
	&=& \frac{\Var(Y)\left(2\Cov(Y,\delta^X_N)+\Cov(\delta_N,\delta^X_N)\right)-\Cov(Y,Y^X)\left(2\Cov(Y,\delta_N)+\Var(\delta_N)\right)}{\Var(Y)\left(\Var(Y)+2\Cov(Y,\delta_N)+\Var(\delta_N)\right)}\\
	&=&\frac{\Var(\delta_N)^{1/2}C_{\delta_N}}{\Var(Y)+2\Cov(Y,\delta_N)+\Var(\delta_N)}
\end{eqnarray*}
and:
\[ \Var(Y)+2\Cov(Y,\delta_N)+\Var(\delta_N) = \Var(Y) + o(1). \]

Finally, $C_{\delta,N}$ is uniformly bounded because $\Var(\delta_N)$ goes to 0 and $\Var(Y)$ is a constant.
\end{proof}

\begin{proof}[Proof of Proposition \ref{prop:aem}]
We will use the following lemma.
\begin{lemma}
	For all $N\in\N^*$, let $(Z_{N,i})_{i=1,\ldots,N}$ be a sequence of i.i.d variables such that
	\begin{enumerate}
		\item $\sqrt N \E(Z_{N,i}) \underset{N\rightarrow+\infty}{\longrightarrow} 0$;
		\item $\Var (Z_{N,i}) \underset{N\rightarrow+\infty}{\longrightarrow} 0$.
	\end{enumerate}
Then
	\[ \frac{1}{\sqrt N} \sum_{i=1}^N Z_{N,i} \overset{\mathbb{P}}{\underset{N\rightarrow+\infty}{\longrightarrow}} 0. \]
\end{lemma}

The lemma follows after the following decomposition:
\[ \frac{1}{\sqrt N} \sum_{i=1}^N Z_{N,i} = \sqrt N \left( \frac{1}{N} \sum_{i=1}^N Z_{N,i} - \E(Z_{N,1}) \right) + \sqrt N \E(Z_{N,1}). \]

Let $U_N$ and $U$ be the vectors defined in the proof of Proposition \ref{prop:ae}, in \eqref{e:defun} and \eqref{e:defu}, respectively, and:
\[ \widetilde U_N = \left( \frac{1}{N} \sum_{i=1}^N \widetilde Y_i \widetilde Y_i^X, \quad
  \frac{1}{N} \sum_{i=1}^N \frac{\widetilde Y_i + \widetilde Y_i^X}{2}, \quad
  \frac{1}{N} \sum_{i=1}^N \frac{\widetilde Y_i^2+(\widetilde Y_i^X)^2}{2} \right). \]
We will show that:
\begin{equation}
\label{e:conve}
\sqrt{N}\left( U_N - \widetilde U_N \right)\overset{\mathbb{P}}{\underset{N\to\infty}{\rightarrow}} 0.
\end{equation}
By Theorem 25.23 of \cite{van2000asymptotic} and the fact that $\left(U_N\right)$ is asymptotically efficient for $U$ (shown in the proof of Proposition \ref{prop:ae}), this implies that $ \left(\widetilde U_N \right) $ is asymptotically efficient for $U$, and the end of the proof of Proposition \ref{prop:ae} shows the announced result.

To prove \eqref{e:conve}, it is sufficient to prove componentwise convergence. We will treat the second and the third components, as the result holds in the same way for the other. 

For the second component, we have 
\[ \frac{1}{\sqrt N} \sum_{i=1}^N (\widetilde Y_{N,i} - Y_i) = \frac{1}{\sqrt N} \sum_{i=1}^N \delta_{N,i} \]
goes to 0 (in probability) by the previous lemma. The same holds for $ \frac{1}{\sqrt N} \sum_{i=1}^N (\widetilde Y_{N,i}^X - Y_i^X) $.

For the third component, we have 
\[  \frac{1}{\sqrt N} \sum_{i=1}^N (\widetilde Y_{N,i}^2 - Y_i^2) = 2 \frac{1}{\sqrt N} \sum_{i=1}^N \delta_{N,i} Y_i + \frac{1}{\sqrt N} \sum_{i=1}^N \delta_{N,i}^2. \]
Now by assumption,
\[ \sqrt N \E(\delta_{N,i} Y_i)\leq\sqrt{N \E(\delta_{N,i}^2)\E(Y_i^2)}=\sqrt{N (\Var(\delta_{N,i})+\E(\delta_{N,i})^2) \E(Y_i^2)}\rightarrow 0, \]  
and by Cauchy-Schwarz inequality,
\[ \Var(\delta_{N,i} Y_i) = \E(\delta_{N,i}^2 Y_i^2) - ( \E(\delta_{N,i} Y_i) )^2 \leq \sqrt{ \E(\delta_{N,i}^4) \E(Y_i^4) } + \E(\delta_{N,i}^2)\E(Y_i^2)
\leq C\E(\delta_N^4)^{1/2}. \]

By assumption, for all $i$, $\delta_{N,i} \overset{\mathbb{P}}{\underset{N \rightarrow +\infty}\longrightarrow} 0$. Hence, the same convergence holds about $\delta_{N,i}^4$. Since $\delta_N$ is in $L^{4+s}$, then  $\left\{\delta_N^4\right\}_N$ is uniformly integrable and we get the convergence of $\E(\delta_N^4)$ to 0 when $N \rightarrow +\infty$.

We conclude by the lemma above. Again, the same convergence occurs for $ \frac{1}{\sqrt N} \sum_{i=1}^N ((\widetilde Y_{N,i}^X)^2 - (Y_i^X)^2) $.
\end{proof}

\section*{Conclusion}
We have shown that the Sobol index estimator considered in this paper is asymptotically normal and asymptotically efficient. We also proved that these two properties are robust with respect to the use of a perturbated model, provided that the perturbation has a variance which decays fast enough. The asymptotic normality property can be used to produce approximate confidence intervals for the Sobol indices; we have presented numerical experiments asserting the reliability of these confidence intervals.

\thanks{This work has been partially supported by the French National
Research Agency (ANR) through COSINUS program (project COSTA-BRAVA
nr. ANR-09-COSI-015). }

\bibliographystyle{plain}
\bibliography{biblio}

\end{document}